\documentclass[reqno,12pt]{amsart}

\usepackage{setspace}

\usepackage{chngcntr}
\counterwithout{figure}{section}
\counterwithout{figure}{subsection}

\usepackage{amsfonts}
\usepackage{amssymb}
\usepackage{graphicx}
\usepackage{pstricks}
\usepackage{amsmath}
\usepackage{amsxtra}
\usepackage{mathtools}
\usepackage{ulem}
\usepackage{url}

\newrgbcolor{lightorange}{1 .5 0}
\newrgbcolor{lightgreen}{0.8 1 .8}
\newrgbcolor{lightyellow}{1 1 0}
\newrgbcolor{lightblue}{0 1 1}
\newrgbcolor{darkgreen}{.1 .5 0}

\usepackage[noadjust]{cite}

\def\TT{\mathbf{T}}

\theoremstyle{plain}
\newtheorem{theorem}{Theorem}[section]
\newtheorem{lemma}[theorem]{Lemma}

\newtheorem{proposition}[theorem]{Proposition}
\newtheorem{corollary}[theorem]{Corollary}

\newtheorem{definition}[theorem]{Definition}

\newtheorem{problem}[theorem]{Problem}

\theoremstyle{definition}

\newtheorem{example}[theorem]{Example}
\newtheorem{remark}[theorem]{Remark}

\numberwithin{equation}{section} 
\def\card{\operatorname{card}}
\def\supp{\operatorname{supp}}

\def\shift{\operatorname{shift}}

\begin{document}

\title{Polynomial embeddings of unilateral weighted shifts into $2$-variable weighted shifts}

\author{Ra\'ul E. Curto }
\address{Department of Mathematics, University of Iowa, Iowa City, Iowa 52242, USA}
\email{raul-curto@uiowa.edu}

\author{Sang Hoon Lee}
\address{Department of Mathematics, Chungnam National University, Daejeon,
34134, Republic of Korea}
\email{slee@cnu.ac.kr}

\author{Jasang Yoon}
\address{School of Mathematical and Statistical Sciences, The University of
Texas Rio Grande Valley, Edinburg, Texas 78539, USA}
\email{jasang.yoon@utrgv.edu}

\thanks{The second author of this paper was partially supported by NRF
(Korea) grant No. 2020R1A2C1A0100584611.}
\thanks{The third named author was partially supported by a grant from the
University of Texas System and the Consejo Nacional de Ciencia y Tecnolog%
\'{\i}a de M\'{e}xico (CONACYT)}

\date{}

\begin{abstract}
Given a bounded sequence $\omega$ of positive numbers, and its associated unilateral weighted shift $W_{\omega}$ acting on the Hilbert space $\ell^2(\mathbb{Z}_+)$, we consider natural representations of $W_{\omega}$ as a $2$-variable weighted shift, acting on the Hilbert space $\ell^2(\mathbb{Z}_+^2)$. \ Alternatively, we seek to examine the various ways in which the sequence $\omega$ can give rise to a $2$-variable weight diagram, corresponding to a $2$-variable weighted shift. \ Our best (and more general) embedding arises from looking at two polynomials $p$ and $q$ nonnegative on a closed interval $I \subseteq \mathbb{R}_+$ and the double-indexed moment sequence $\{\int p(r)^k q(r)^{\ell} d\sigma(r)\}_{k,\ell \in \mathbb{Z}_+}$, where $W_{\omega}$ is assumed to be subnormal with Berger measure $\sigma$ such that $\supp \; \sigma \subseteq I$; we call such an embedding a $(p,q)$-embedding of $W_{\omega}$. \ We prove that every $(p,q)$-embedding of a subnormal weighted shift $W_{\omega}$ is (jointly) subnormal, and we explicitly compute its Berger measure. \ We apply this result to answer three outstanding questions: (i) Can the Bergman shift $A_2$ be embedded in a subnormal $2$-variable spherically isometric weighted shift $W_{(\alpha,\beta)}$? \ If so, what is the Berger measure of $W_{(\alpha,\beta)}$? \ (ii) Can a contractive subnormal unilateral weighted shift be always embedded in a spherically isometric $2$-variable weighted shift? \ (iii) Does there exist a (jointly) hyponormal $2$-variable weighted shift $\Theta(W_{\omega})$ (where $\Theta(W_{\omega})$ denotes the classical embedding of a hyponormal unilateral weighted shift $W_{\omega}$) such that some integer power of $\Theta(W_{\omega})$ is not hyponormal? \ As another application, we find an alternative way to compute the Berger measure of the Agler $j$-th shift $A_{j}$ ($j\geq 2$). \ Our research uses techniques from the theory of disintegration of measures, Riesz functionals, and the functional calculus for the columns of the moment matrix associated to a polynomial embedding.
\end{abstract}

\maketitle

\setcounter{tocdepth}{1}

\tableofcontents


\doublespacing

\section{Introduction}

Given a bounded sequence $\omega$ of positive numbers, and its associated unilateral weighted shift $W_{\omega}$ acting on the Hilbert space $\ell^2(\mathbb{Z}_+)$, we consider natural representations of $W_{\omega}$ as a $2$-variable weighted shift, acting on the Hilbert space $\ell^2(\mathbb{Z}_+^2)$. \ Alternatively, we seek to examine the various ways in which the sequence $\omega$ can give rise to a $2$-variable weight diagram, corresponding to a $2$-variable weighted shift $W_{(\alpha,\beta)}$.\footnote{\ Hereafter, $\mathbb{Z}_+$ denotes the set of nonnegative integers $0,1,2,3,\cdots$.} \ Our best (and more general) embedding arises from looking at two polynomials $p$ and $q$ nonnegative on a closed interval $I \subseteq \mathbb{R}_+$ and the double-indexed moment sequence $\{\int p(r)^k q(r)^{\ell} d\sigma(r)\}_{k,\ell \in \mathbb{Z}_+}$, where $W_{\omega}$ is assumed to be subnormal with Berger measure $\sigma$ such that $\supp \; \sigma \subseteq I$; we call such an embedding a $(p,q)$-embedding of $W_{\omega}$, and denoted as $W_{(\alpha(\omega),\beta(\omega))}$. \ When $p(r) \equiv q(r):=r \quad (r\ge0)$, we recover the \textit{classical embedding} discussed in \cite{CLY2} and \cite{CuYo6}; when $p(r):=r$ and $q(r):=1-r \quad (r \in [0,1])$, we obtain the $2$-variable \textit{spherically isometric embedding}, studied in \cite{CLYJFA}, \cite{CuYo6} and \cite{CuYo7} (here $\alpha_{\mathbf{k}}^2+\beta_{\mathbf{k}}^2=1$ for all $\mathbf{k}\in\mathbb{Z}_+^2$); and when $p(r):=r^2$ and $q(r):=r^3$ we generate a \textit{Neil parabolic embedding}. 

We first prove that every $(p,q)$-embedding of a subnormal weighted shift $W_{\omega}$ is (jointly) subnormal, and we explicitly compute its Berger measure. \ Next, we use polynomial embeddings to give answers to several open problems regarding spherically quasinormal $2$-variable weighted shifts constructed from $W_{\omega }$. \ We do these by appealing to techniques from the theory of disintegration of measures, Riesz functionals, and the functional calculus for the columns of the moment matrix associated to a polynomial embedding. 

In a different direction, we recall that for a unilateral weighted shift $W_{\omega}$ it is well known that the hyponormality of $W_{\omega }$ implies the hyponormality of every integer power $W_{\omega}^{m}\;(m\geq 1)$. \ We use our theory of polynomial embeddings to build a (jointly) hyponormal $2$-variable weighted shift $\Theta(W_{\omega})$ such that some integer power of it is not hyponormal. \ As another application, we find an alternative way to compute the Berger measure of the Agler $j$-th shift $A_{j}$ ($j\geq 2$). \ 

One of the by-products of our research is the fact that some questions pertaining to unilateral weighted shifts (and therefore within the realm of single-variable operator theory) can be answered after rephrasing them in the context of $(p,q)$-embeddings, a topic that rightly belongs in multivariable operator theory.

The organization of the paper is as follows. \ In Section \ref{prelim} we discuss spherical quasinormality, both for commuting pairs of operators and more specifically for $2$-variable weighted shifts. \ These topics require basic knowledge of unilateral and $2$-variable weighted shifts, and of the basic construction of spherically isometric $2$-variable weighted shifts; the interested reader will find a summary of each such topic in the Appendix (Section \ref{appendix}). \ In Section \ref{embed} we introduce polynomial embeddings and we prove one of our main results: every polynomial embedding of a subnormal unilateral weighted shift is subnormal. \ Section \ref{sectionclassical} is devoted to the specific case of the classical embedding. \ We then use the Neil parabolic embedding $r \mapsto (r^2,r^3)$ to study some examples linking the $k$-hyponormality of $W_{\omega}$ to the $k$-hyponormality of the $(2,3)$-power of $\Theta(W_{\omega})$. \ In particular, we use polynomial embeddings to give an alternative proof that hyponormality of $2$-variable weighted shifts is not preserved under powers (cf. \cite{CLY1}, \cite{CLYJFA}). \ Along the way, we exhibit the connection between the respective Berger measures. \ In Section \ref{Bergman} we find a concrete formula for the Berger measure of the spherically isometric embedding of the Bergman shift; this solves Question 5.15 in \cite{CuYo7}. \ In Section \ref{examples}, we first give an alternative solution to Problem \ref{prob1}(ii) that employs techniques from the theory of recursively generated weighted shifts. \ We conclude Section \ref{examples} with two examples that illustrate our solution to Problem \ref{prob1}. \ 

We now present a precise formulation of the central questions of this paper.

\begin{problem}
\label{prob1}\cite{CuYo7} Consider a spherically quasinormal $2$-variable
weighted shift $W_{(\alpha ,\beta )}$ and let $\sigma $ be
the Berger measure of $W_{0}:=\shift (\alpha _{(0,0)},\alpha
_{(1,0)},\cdots )$. \ Since $W_{(\alpha ,\beta )}$ is subnormal (by \cite[Theorem
3.10]{CuYo7}), let $\mu $ be its Berger measure. \medskip \newline
(i) \ Describe $\mu $ in terms of $\sigma $. \medskip \newline
(ii) \ Assume that $W_{0}$ is recursively generated; that is, $\sigma$ is finitely atomic. \ It is known that $\mu
$ is also finitely atomic, and that $\supp \; \mu \subseteq \supp \; \sigma \times (c-\supp \; \sigma )$, where $c>0$ is the constant of spherical quasinormality \cite[Theorem 5.2]{CuYo7}. \ What else can we say? \ Can we
give a concrete formula for the atoms and densities of $\mu $?
\end{problem}

\begin{problem}
\label{quest}\cite{CuYo7} Assume that $W_{0}$ is the Bergman shift $B_+:=\shift(\sqrt{\frac{1}{2}},\sqrt{\frac{2}{3}},
\sqrt{\frac{3}{4}},\cdots )$, and use
Subsection \ref{constr} to build a spherically quasinormal $2$-variable
weighted shift $W_{(\alpha ,\beta )}$ (cf. Figure \ref{Fig 4}). \ Observe that the $j$-th row is identical to the $j$-column, for every $j\geq 0$. \
Note also that $W_{(\alpha ,\beta )}$ is a close relative of the Drury-Arveson $2$-variable
weighted shift, in that the $j$-row of $W_{(\alpha ,\beta )}$ is the Agler $(j+2)$-th shift $A_{j+2}$. \ What is the Berger measure of $W_{(\alpha ,\beta )}$?
\end{problem}


\setlength{\unitlength}{1mm} \psset{unit=1mm}
\begin{figure}[th]
\begin{center}
\begin{picture}(130,70)

\psline{->}(20,14)(35,14)
\put(27,10){$\rm{T}_1$}
\psline{->}(2,35)(2,50)
\put(-3,42){$\rm{T}_2$}

\psline{->}(5,20)(55,20)
\psline(5,35)(53,35)
\psline(5,50)(53,50)
\psline(5,65)(53,65)

\psline{->}(5,20)(5,70)
\psline(20,20)(20,68)
\psline(35,20)(35,68)
\psline(50,20)(50,68)

\put(0,16){\footnotesize{$(0,0)$}}
\put(16,16){\footnotesize{$(1,0)$}}
\put(31,16){\footnotesize{$(2,0)$}}
\put(46,16){\footnotesize{$(3,0)$}}

\put(10,22){\footnotesize{$\sqrt{\frac{1}{2}}$}}
\put(25,22){\footnotesize{$\sqrt{\frac{2}{3}}$}}
\put(40,22){\footnotesize{$\sqrt{\frac{3}{4}}$}}
\put(51,21){\footnotesize{$\cdots$}}

\put(10,37){\footnotesize{$\sqrt{\frac{1}{3}}$}}
\put(25,37){\footnotesize{$\sqrt{\frac{2}{4}}$}}
\put(40,37){\footnotesize{$\sqrt{\frac{3}{5}}$}}
\put(51,36){\footnotesize{$\cdots$}}

\put(10,52){\footnotesize{$\sqrt{\frac{1}{4}}$}}
\put(25,52){\footnotesize{$\sqrt{\frac{2}{5}}$}}
\put(40,52){\footnotesize{$\sqrt{\frac{3}{6}}$}}
\put(51,51){\footnotesize{$\cdots$}}

\put(10,66){\footnotesize{$\cdots$}}
\put(25,66){\footnotesize{$\cdots$}}
\put(40,66){\footnotesize{$\cdots$}}
\put(51,66){\footnotesize{$\cdots$}}

\put(5,26){\footnotesize{$\sqrt{\frac{1}{2}}$}}
\put(5,41){\footnotesize{$\sqrt{\frac{2}{3}}$}}
\put(5,56){\footnotesize{$\sqrt{\frac{3}{4}}$}}
\put(6,66){\footnotesize{$\vdots$}}

\put(20,26){\footnotesize{$\sqrt{\frac{1}{3}}$}}
\put(20,41){\footnotesize{$\sqrt{\frac{2}{4}}$}}
\put(20,56){\footnotesize{$\sqrt{\frac{3}{5}}$}}
\put(21,66){\footnotesize{$\vdots$}}

\put(35,26){\footnotesize{$\sqrt{\frac{1}{4}}$}}
\put(35,41){\footnotesize{$\sqrt{\frac{2}{5}}$}}
\put(35,56){\footnotesize{$\sqrt{\frac{3}{6}}$}}
\put(36,66){\footnotesize{$\vdots$}}


\put(59,65){\footnotesize{The $2$-variable weighted shift $W_{(\alpha,\beta)}$ whose weight}}
\put(59,55){\footnotesize{diagram is shown on the left has the following properties:}}
\put(59,45){\footnotesize{(i)} \ $\alpha_{(k_1,k_2)}^2+\beta_{(k_1,k_2)}^2=1$ for all $(k_1,k_2) \in \mathbb{Z}_+^2.$}
\put(59,35){\footnotesize{(ii)} \ $W_{(\alpha,\beta)}$ \footnotesize{is a spherical isometry.}}
\put(59,25){\footnotesize{(iii)} \ \footnotesize{The weights in the $j$-row correspond to}}
\put(67.5,20){\footnotesize{the Agler $(j+2)$-th shift} $A_{j+2}$.}

\end{picture}
\end{center}
\caption{Weight diagram of the $2$-variable
weighted shift in Problem \ref{quest}.}
\label{Fig 4}
\end{figure}
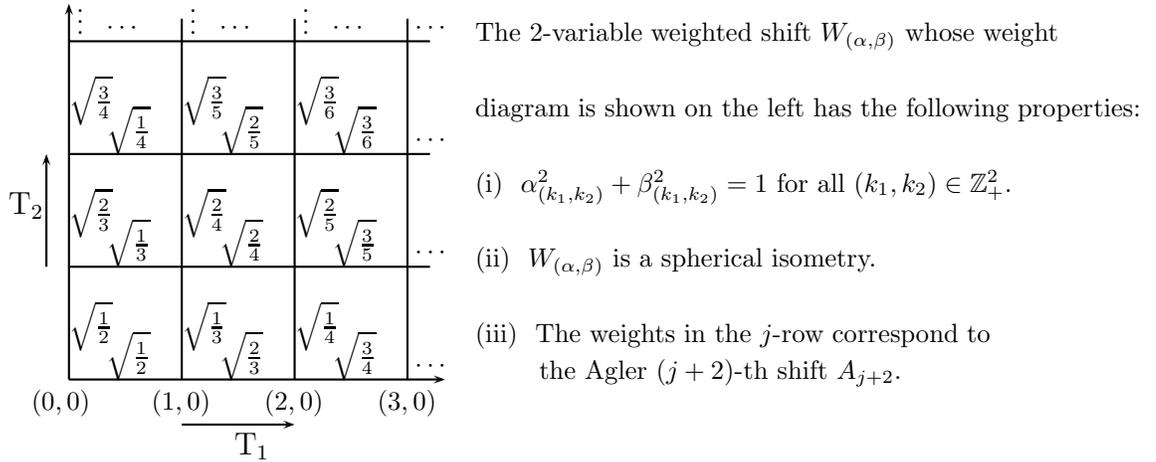

It is now natural to ask:

\begin{problem}
\label{Q1} Given a subnormal unilateral weighted shift $W_{\omega }$, can we always construct
a spherically quasinormal $2$-variable weighted shifts $W_{\left( \alpha(\omega)
,\beta(\omega) \right) }$? \ In those instances when this is possible, we shall use the notation $\mathcal{SQNE}(\omega)$ for the spherically quasinormal embedding.  
\end{problem}

For the next question, recall that a subnormal unilateral weighted shift is recursively generated if and only if its Berger measure is finitely atomic (cf. Subsection \ref{RGWS} in the Appendix).

\begin{problem}
\label{prob2}(i) Let $W_{\omega }$ be a recursively generated weighted shift, and let $m$ be a positive integer. \ Is $W_{\omega }^{m}$ a direct sum of recursively generated weighted shifts?\medskip \newline
(ii) Conversely, assume that $W_{\omega }^{m}$ is a direct sum of recursively
generated weighted shifts for some $m\geq 1$. \ Is $W_{\omega }$ a
recursively generated weighted shift?
\end{problem}

We will prove that Problem \ref{prob2} (i) has a positive answer, but (ii) does not. \ In
fact, $W_{\omega }$ may not even be subnormal. \ We can then ask, under what conditions does Problem \ref{prob2}(ii) yield a positive answer?

For a subnormal $2$-variable weighted shift $W_{(\alpha ,\beta
)}$ it is natural to say that it is \textit{recursively generated} if its Berger measure is finitely atomic. \ We shall see in Theorem \ref{Prop1} that the property of being recursively generated transfers from $W_{\omega}$ to any $(p,q)$-embedding. \ We plan to build on this fact and study recursively generated $2$-variable weighted shifts in detail in a forthcoming paper.

\begin{problem} \label{prob5}
Let $W_{\omega}$ be a hyponormal unilateral weighted shift, and let $W_{\theta(\omega)} \equiv \Theta(W_{\omega})$ denote the $2$-variable weighted shift obtained via the classical embedding, that is, the $(p,q)$-embedding with $p(r) \equiv q(r):=r \; (r \in [0,\left\|W_{\omega}\right\|^2])$. \ Also, let $m$ and $n$ be two positive integers, and consider the power $W_{\theta(\omega)}^{(m,n)}$. \ Is $W_{\theta(\omega)}^{(m,n)}$ hyponormal?
\end{problem}
(Hereafter, we will ordinarily use the notation $W_{\theta(\omega)}$, for three reasons: (i) it easily signals the action of the classical embedding on $1$-variable sequences; (ii) it more clearly denotes that it is a weighted shift; and (iii) it is more concise that allows for a succinct display of the powers.) \ 

\medskip
We conclude this section by stating our main results, which provide answers to the above mentioned problems. 

\begin{theorem}
\label{Th4}Let $W_{\left( \alpha(\omega) ,\beta(\omega) \right) }$ be a $(p,q)$-embedding of
$W_{\omega }$, and let $f(r):=(p(r),q(r))$. \ If $W_{\omega }$ is subnormal with Berger measure $\sigma $ such that 
$\left. p \right|_{\supp \; \sigma} \ge 0$ and $\left. q\right|_{\supp \; \sigma} \ge 0$, then $W_{\left( \alpha(\omega) ,\beta(\omega) \right) }$ is subnormal, with Berger measure $\mu:=\sigma \circ f^{-1}$.
\end{theorem}

\begin{theorem} \label{thmBergman}
The Bergman shift $B_+$ admits a spherically isometric embedding whose Berger measure is normalized arclength on the line segment $\{(s,t)\in\mathbb{R}_+^2:s+t=1\}$. 
\end{theorem}

\begin{theorem}
\label{Prop1}Let $W_{\left( \alpha(\omega),\beta(\omega) \right) }$ be a $(p,q)$-embedding of $W_{\omega}$, where $W_{\omega }$ is
subnormal with finitely atomic Berger measure $\sigma =\sum_{i=0}^{\ell
-1}\rho _{i}\delta _{r_{i}}$, $\sum_{i=0}^{\ell -1}\rho _{i}=1$, $\rho _{i}>0
$, and $0\leq r_{0}<r_{1}<\cdots<r_{\ell-1}$. \ Then, the Berger measure $\mu $ of $
W_{\left( \alpha(\omega) ,\beta(\omega) \right) }$ is
\begin{equation}
\mu =\sum_{i=0}^{\ell -1}\rho _{i}\delta _{\left(p(r_{i}),q(r_{i})\right)
}.  \label{finite measure}
\end{equation}
\end{theorem}

As a consequence, we can generalize Theorem \ref{thmBergman}, as follows.
 
\begin{theorem} \label{arclength}
Let $W_{(\alpha(\omega),\beta(\omega))}$ be a $(p,q)$-embedding of a subnormal weighted shift $W_{\omega}$, and denote by $\mu$ and $\sigma$ the Berger measures of $W_{(\alpha(\omega),\beta(\omega))}$ and $W_{\omega}$, respectively. \ Assume that $\sigma$ is the Lebesgue measure on $[0,1]$. \ Then $d\mu(p(r),q(r))$ is normalized arclength on the curve $C:=\{(p(r),q(r)):0 \le r \le 1\}$.  
\end{theorem}

Related to Problem \ref{prob2}, we have:

\begin{theorem}
\label{Thm2A}If $W_{\omega }$ is subnormal and if $W_{\omega }^{m}$ is a
direct sum of recursively generated weighted shifts, then $W_{\omega }$ is a
recursively generated weighted shift.
\end{theorem}


\section{Notation and Preliminaries} \label{prelim}

\subsection{Spherically quasinormal $n$-tuples of operators} \ Let $\mathcal{H}$ be a complex Hilbert space and let $\mathcal{B}(\mathcal{H}%
)$ denote the algebra of all bounded linear operators on $\mathcal{H}$. 
$\ $%
We say that $T\in \mathcal{B}(\mathcal{H})$ is \textit{normal} if $T^{\ast
}T=TT^{\ast }$, \textit{quasinormal} if $T$ commutes with $T^{\ast }T$,
i.e., $TT^{\ast }T=T^{\ast }T^{2}$, \textit{subnormal} if $T=\left.N\right|_{\mathcal{H}%
} $, where $N$ is normal and $N(\mathcal{H})\subseteq \mathcal{H}$, and
\textit{hyponormal} if $T^{\ast }T\geq TT^{\ast }$.\ \ For the $1$-variable
case, we clearly have that%
\begin{equation*}
\begin{tabular}{l}
normal $\Longrightarrow $ quasinormal $\Longrightarrow $ subnormal $%
\Longrightarrow $ hyponormal.%
\end{tabular}%
\end{equation*}%
\ For $S,T\in \mathcal{B}(\mathcal{H})$ let $[S,T]:=ST-TS$. \ We say that an
$n$-tuple $\mathbf{T}\equiv(T_{1},\cdots ,T_{n})$ of operators on $\mathcal{H}$
is (jointly) \textit{hyponormal} if the operator matrix
\begin{equation*}
\lbrack \mathbf{T}^{\ast },\mathbf{T]:=}\left(
\begin{array}{lll}
\lbrack T_{1}^{\ast },T_{1}] & \cdots & [T_{n}^{\ast },T_{1}] \\
\text{ \thinspace \thinspace \quad }\vdots & \ddots & \text{ \thinspace
\thinspace \quad }\vdots \\
\lbrack T_{1}^{\ast },T_{n}] & \cdots & [T_{n}^{\ast },T_{n}]%
\end{array}%
\right)
\end{equation*}%
is positive on the direct sum of $n$ copies of $\mathcal{H}$ (cf. \cite{Ath}%
, \cite{bridge}). \ The $n$-tuple $\mathbf{T}$ is said to be \textit{normal}
if $\mathbf{T}$ is commuting and each $T_{i}$ is normal and \textit{subnormal }if $\mathbf{T}$ is the restriction of a normal $n$-tuple to a common invariant subspace. \ A commuting $n$-tuple $\mathbf{T}$ is $k$-hyponormal if the tuple 
$(T_1,\cdots,T_n,T_1^2,T_1T_2,\cdots,T_n^2,\cdots,T_1^k,T_1^{k-1}T_2,\cdots,T_n^k)$ is jointly hyponormal.
 
\medskip
Following A. Athavale, S. Podder \cite{AtPo} and J. Gleason \cite{Gle}, we say that \newline
(i) \ $\TT$ is {\it matricially quasinormal} if $T_i$ commutes with $T_j^*T_k$ for all $i,j,k = 1,\cdots,n$; \newline
(ii) \ $\TT$ is {\it (jointly) quasinormal} if $T_i$ commutes with $T_j^*T_j$ for all $i,j = 1,\cdots,n$; and \newline 
(iii) \ $\TT$ is {\it spherically quasinormal} if $T_i$ commutes with 
$$
P:=T_1^*T_1+\cdots+T_n^*T_n,
$$
for $i=1,\cdots,n$. \ Then 
\medskip
\begin{eqnarray}
\textrm{normal } & \implies & \textrm{ matricially quasinormal } \implies  \textrm{ (jointly) quasinormal  } \nonumber \\
& \implies & \textrm{ spherically quasinormal } \implies \textrm{ subnormal (\cite[Proposition 2.1]{AtPo})} \nonumber \\
& \implies & k \textrm{-hyponormal} \implies \textrm{hyponormal.}
\end{eqnarray} \label{implication}

On the other hand, results of R.E. Curto, S.H. Lee and J. Yoon (cf. \cite{CLYJFA}), and of J. Gleason \cite{Gle} show that the reverse implications in (\ref{implication}) do not necessarily hold.

In \cite[Theorem 2.2]{CuYo7}, R.E. Curto and J. Yoon showed that the spherically quasinormal commuting pairs are precisely the fixed points of the spherical Aluthge transform. \ In \cite{CuYo6} it was also shown that every spherically quasinormal $2$-variable weighted shift is a positive multiple of a spherical isometry (see Theorem \ref{Quasinormal}). \ In order to state this result, we need a brief discussion of unilateral and $2$-variable weighted shifts; we refer the reader to Subsections \ref{sub25} and \ref{sub30} of the Appendix.


\subsection{Spherically Quasinormal $2$-variable Weighted Shifts}

In this section we present a characterization of spherical quasinormality for $2$-variable
weighted shifts, which was announced in \cite{CuYo4} and proved in \cite{CuYo6} and \cite{CuYo7}. \ Before we state it, we list some simple facts about quasinormality for $2$-variable weighted shifts.

\begin{remark} (cf. \cite{CLYJFA}) \ We first observe that no $2$-variable weighted shift can be matricially quasinormal, as a simple calculation shows. \ Also, a $2$-variable weighted shift $W_{(\alpha,\beta)}$ is (jointly) quasinormal if and only if $\alpha_{(k_1,k_2)}=\alpha_{(0,0)}$ and $\beta_{(k_1,k_2)}=\beta_{(0,0)}$ for all $k_1,k_2 \ge 0$. \ This can be seen via a simple application of (\ref{commuting}) and (\ref{adjoint}). \ As a result, up to a scalar multiple in each component, a quasinormal $2$-variable weighted shift is identical to the so-called Helton-Howe shift. \ This fact is entirely consistent with the one-variable result: a unilateral weighted shift $W_{\omega}$ is quasinormal if and only if $W_{\omega}=c \cdot U_+$ for some $c>0$. 
\end{remark}

We now recall the class of spherically isometric commuting pairs of
operators (cf. \cite{Ath1}, \cite{AtPo}, \cite{Gle}). \ A commuting pair 
$\mathbf{T}\equiv (T_{1},T_{2})$ is a
{\it spherical isometry} if $T_{1}^{\ast }T_{1}+T_{2}^{\ast }T_{2}=I$.

\begin{lemma}
\label{Sp-sub} \cite[Proposition 2]{Ath1} \ Any spherical isometry is subnormal.
\end{lemma}

\begin{theorem}
(\cite[Theorem 3.1]{CLYJFA}; cf. \cite[Lemma 10.3]{CuYo6}) \  
\label{Quasinormal} \ For a commuting $2$-variable weighted shift $W_{\left( \alpha
,\beta \right) }\equiv\left( T_{1},T_{2}\right) $, the following statements are
equivalent:\newline
(i) $\ W_{\left( \alpha ,\beta \right) }$ is a spherically quasinormal $2$-variable weighted shift; \newline
(ii) \ (algebraic condition) \ $T_{1}^{\ast }T_{1}+T_{2}^{\ast }T_{2}=c I$, for some constant $c$; \newline
(iii) \ (weight condition) \ for all $(k_{1},k_{2})\in \mathbb{Z}_{+}^{2}$, $%
\alpha _{(k_{1},k_{2})}^{2}+\beta _{(k_{1},k_{2})}^{2}=c$, for some constant $c>0$; \newline
(iv) \ (moment condition) \ for all $(k_{1},k_{2})\in \mathbb{Z}_{+}^{2}$, 
$\gamma_{(k_1+1,k_2)}+\gamma_{(k_1,k_2+1)}=c\gamma_{(k_1,k_2)}$, for some constant $c>0$.
\end{theorem}

\begin{corollary} \label{subnormal}
Any spherically quasinormal $2$-variable weighted shift is subnormal.
\end{corollary}

It is now natural to ask:

\begin{problem}
\label{Q1A} Given a subnormal unilateral weighted shift $W_{\omega }$, can we always construct
a spherically quasinormal $2$-variable weighted shifts $W_{\left( \alpha(\omega),\beta(\omega) \right) }$?   
\end{problem}

In particular, if $W_{\omega }=B_{+}$, the Bergman
shift, or if $W_{\omega }=\shift (a,b,b,\cdots )$ with $0<a<b<1$, can we construct a spherically quasinormal $2$-variable weighted shift $%
W_{\left( \alpha(\omega) ,\beta(\omega) \right) }$? \ For the case of the Bergman shift, Figure \ref{Fig 4} shows the desired spherical isometry. \ Thus, in the case of the Bergman shift, we can obtain a spherically isometric embedding, which we will denote by $\mathcal{SIE}(\omega)$. \ While this can be established via a detailed calculation, we shall prove in Section \ref{embed} a much more general result, linking concretely the Berger measures of $W_{\omega}$ and $%
W_{\left( \alpha(\omega) ,\beta(\omega) \right) }$. \ For the second shift, observe that its Berger measure is $(1-(\frac{a}{b})^2)\delta_0+(\frac{a}{b})^2\delta_{b^2}$. \ From this one can show that $W_{\omega}$ does admit a spherically isometric extension, whose Berger measure is $(1-(\frac{a}{b})^2)\delta_{(0,1)}+(\frac{a}{b})^2\delta_{(b^2,1-b^2)}$.  


\subsection{Construction of Spherically Quasinormal $2$-variable Weighted Shifts} \label{constr}

As observed in \cite[Section 4]{CuYo7}, within the class of $2$-variable weighted shifts there is a simple description of spherical isometries, in terms of the weight sequences $\alpha \equiv \{\alpha_{(k_1,k_2)}\}$ and $\beta\equiv \{\beta_{(k_1,k_2)}\}$. \ Indeed, since spherical isometries are (jointly) subnormal, we know that the unilateral weighted shift associated with each row is subnormal, and in particular the $0$-th row in the weight diagram corresponds to a contractive subnormal unilateral weighted shift; let us denote its weights by $\omega \equiv \{\alpha_{(k,0)}\}_{k=0,1,2,\cdots}$ and its Berger measure by $\sigma$. \ Also, in view of Theorem \ref{Quasinormal} we can assume that $c=1$. \ We may also assume, without loss of generality, that $\omega$ is strictly increasing. \ Using the identity 
\begin{equation} \label{sphericalidentity}
\alpha_{\mathbf{k}}^2+\beta_{\mathbf{k}}^2=1 \; \; \; (\mathbf{k} \in \mathbb{Z}_+^2),
\end{equation}
we immediately see that $\beta_{(k,0)}=\sqrt{1-\alpha_{k,0}^2}$ for $k=0,1,2,\cdots$. \ With these new values at our disposal, we can use the commutativity property (\ref{commuting}) to generate the values of $\alpha$ in the first row; that is, 
\begin{equation} \label{alpharelation}
\alpha_{(k,1)}=\alpha_{(k,0)}\beta_{(k+1,0)}/\beta_{(k,0)}.
\end{equation}
Let us briefly pause to observe that, starting with the sequence $\{\alpha_{(k,0)}\}_{k \ge 0}$ for the $0$-th row, we have been able to generate the sequence $\tau \equiv \{\alpha_{(k,1)}\}_{k \ge 0}$ for the first row. \ Now, the $0$-row sequence came with a Berger measure $\sigma$; does the new sequence have a Berger measure? \ Let us look at the moments of $\tau$: $\gamma_0(W_{\tau})=1$ and, for $k \ge 1$, 
\begin{eqnarray*}
\gamma_k(W_{\tau})&=&\alpha_{(0,1)}^2\alpha_{(1,1)}^2 \cdots \alpha_{(k-1,1)}^2 \\
											&=& \frac{1}{\beta_{(0,0)}^2}\alpha_{(0,0)}^2\alpha_{(1,0)}^2 \cdots \alpha_{(k-1,0)}^2\beta_{(k,0)}^2 \quad (\textrm{by }(\ref{alpharelation})) \\
											&=& \frac{1}{\beta_{(0,0)}^2}\gamma_k(W_{\omega})\beta_{(k,0)}^2 \\
											&=& \frac{1}{\beta_{(0,0)}^2}\gamma_k(W_{\omega})(1-\alpha_{(k,0)}^2) \quad (\textrm{by }(\ref{sphericalidentity})) \\
											&=& \frac{1}{\beta_{(0,0)}^2}(\gamma_k(W_{\omega})-\gamma_{k+1}(W_{\omega})) \\
											&=& \frac{1}{\beta_{(0,0)}^2}(\int r^k d\sigma(r) - \int r^{k+1} d\sigma(r)) \\
											&=& \frac{1}{\beta_{(0,0)}^2} \int r^k(1-r) d\sigma(r).
\end{eqnarray*}
It is now evident that $\tau$ does have a Berger measure $\zeta$, given by $d\zeta(r) = \frac{1}{\beta_{(0,0)}^2} (1-r) d\sigma(r)$.
\ This is entirely consistent with the fact that in any subnormal $2$-variable weighted shift with Berger measure 
$d\mu(s,t)$, the Berger measures $\xi_1$ and $\xi_0$ of the first and $0$-th rows, respectively, are related by the equation 
$d\xi_1(s)=\frac{1}{\gamma{(0,1)}}t\xi_0(s)$. \ Also, note that for the new sequence, $\alpha_{(k,1)} < \alpha_{(k,0)} < 1$ for all $k \ge 0$.

We can now repeat the algorithm, and calculate the weights $\beta_{(k,1)}$ for $k=0,1,2,\cdots$, again using the identity (\ref{sphericalidentity}). \ This in turn leads to the $\alpha$ weights for the second row, and so on. \ For more on this construction, the reader is referred to \cite{CuYo7}. \ In particular, it is worth noting that the construction may stall if the sequence $\{\alpha_{(k,0)}\}_{k \ge 0}$ is not strictly increasing (see Proposition \ref{Prop100}). \ Moreover, the fact that we can determine concretely the Berger measures of each row (and each column, if we were to start with the $0$-column and apply the algorithm!) does not guarantee the existence of a Berger measure for the $2$-variable weighted shift. \ That comes from either Corollary \ref{subnormal} or the more general Theorem \ref{Th4} (proved in Section \ref{embed}). \ As we know from \cite{CuYo2}, it is possible for a commuting $2$-variable weighted shift to have rows and columns with mutually absolutely continuous Berger measures, and still not be (jointly) subnormal. \ Thus, in the case of $2$-variable weighted shifts, the Lifting Problem for Commuting Subnormals entails more than the subnormality of rows and columns (see for instance \cite{Cur}, \cite{CuYo1}, \cite{CuYo2}).  

\begin{proposition} \label{Prop100} (\cite[Proposition 12.14]{Cur})
Let 
$$
\alpha_{(0,0)}:=\sqrt{p},  \; \; \alpha_{(1,0)}:=\sqrt{q}, \; \;  \alpha_{(2,0)}:=\sqrt{r}  \; \; \textrm{ and } \; \alpha_{(3,0)}:=\sqrt{r},
$$
 and assume that $0<p<q<r<1$. \  Then the algorithm described in this section fails at some stage.  \ As a consequence, there does not exist a spherical isometry interpolating these initial data.
\end{proposition}

\begin{remark}
(a) \ In Proposition \ref{Prop100} the reader may have noticed that the $0$-th row is not subnormal; for, it is well known that, up to a constant, the only subnormal unilateral weighted shifts with two equal weights are $U_+$ and $S_a$ (\cite[Theorem 6]{Sta}). \ Thus, save for these two special (trivial) cases, assuming subnormality of the $0$-th row will automatically guarantee that $\alpha_{(k,0)}$ is strictly increasing; therefore, in the sequel we will always assume that the $0$-th row is subnormal. \newline
(b) \ While the $0$-th row of the shift in Proposition \ref{Prop100} is not subnormal (not even quadratically hyponormal!), it is possible to prove that the algorithm also stalls with a $0$-th row of the form $\frac{3}{4}<\sqrt{\frac{2}{3}}<\sqrt{\frac{3}{4}}<\sqrt{\frac{4}{5}}<\cdots$, which is a rank-one perturbation of $B_+$. \ We shall prove in Theorem \ref{thm51} that the construction algorithm always works when the $0$-th row is subnormal.
\end{remark} \


\section{Polynomial Embeddings of Unilateral Weighted Shifts} \label{embed}

Let $I$ be a closed interval in $\mathbb{R}_+$. \ Consider a one-parametric representation of a planar curve $C$ described by a function $f:I \longrightarrow \mathbb{R}_{+}^{2}$ given by $f\left( r\right)
\equiv \left( s,t\right) :=\left( p\left( r\right) ,q\left( r\right) \right) $,
where $p,q$ are nonnegative real polynomials on $I$ and $f$
is injective. \ Assume that $W_{\omega }$ is a subnormal unilateral weighted
shift with Berger measure $\sigma $ and such that $\supp \; \sigma \subseteq I$, that is,
\begin{equation}
\gamma _{k}\left( W_{\omega }\right):=\omega _{0}^{2}\omega _{1}^{2}\cdots \omega _{k-1}^{2} =\int_{I}r^{k}d\sigma
(r).  \label{moment2}
\end{equation}%

\begin{definition}
\label{newdef} The polynomial embedding $\mathcal{PE}(\omega;f)$ of $W_{\omega}$ via the pair $f \equiv (p,q)$ is the $2$-variable weighted shift $W_{(\alpha(\omega),\beta(\omega))}$ with moments
\begin{equation} \label{moment1}
\gamma _{(k_{1},k_{2})}:=\int_{I}p(r)^{k_{1}}q(r)^{k_{2}}d\sigma (r).
\end{equation}
\end{definition}

If we knew that $\mathcal{PE}(\omega;f)$ is subnormal, with Berger measure $\mu$, the moment equations would be
$$
\int s^{k_1}t^{k_2} d\mu(s,t)=\gamma_{(k_1,k_2)}[\mu]=\int_{I}p(r)^{k_1} q(r)^{k_2}d\sigma (r)
=\int s^{k_1}t^{k_2} d\sigma(f^{-1}(s,t)),
$$
after using the change of variables $s=p(r),\; t=q(r)$. \ This gives the heuristics for the following result.

\begin{theorem}
\label{Th4A}(cf. Problems \ref{prob1} and \ref{Q1}) \ Let $W_{\left( \alpha(\omega) ,\beta(\omega) \right) }$ be a $(p,q)$-embedding of
$W_{\omega }$. \ If $W_{\omega }$ is subnormal with Berger measure $\sigma $%
, then $W_{\left( \alpha(\omega) ,\beta(\omega) \right) }$ is subnormal, with Berger measure $\mu:=\sigma \circ f^{-1}$.
\end{theorem}

\begin{proof}
Define a probability measure $\mu$ on $\mathbb{R}_+^2$ by $\mu(E):=\sigma(f^{-1}(E))$, for all Borel sets $E$ in $\mathbb{R}_+^2$. \ Now compute the moments of $\mu$ and compare them to	 those in \ref{moment1}.
\end{proof}

\noindent
{\bf The Classical Embedding}. \ We now recall the $2$-variable weighted shift embedding introduced and studied in \cite{CLY2} and \cite{CuYo6}: \ Given a $1$%
-variable unilateral weighted shift $W_{\omega }$ we embed $\omega $ into $\ell
^{\infty}(\mathbb{Z}_{+}^{2})$ as follows:
\begin{equation}
\alpha _{(k_{1},k_{2})}\equiv \beta _{(k_{1},k_{2})}:=\omega
_{k_{1}+k_{2}}\;\;(k_{1},k_{2}\geq 0).  \label{eq11}
\end{equation}%
We denote the associated $2$-variable weighted by $W_{\theta(\omega)}$
(see Figure \ref{Fig 3}(i)); we will also use the notation $\mathcal{CE}(\omega)$. \ We also recall that a probability measure
 $\epsilon $ on $X\times X$ is said to be \textit{diagonal} if $\supp \; \epsilon
\subseteq \{(s,s):s\in X\}$.



\setlength{\unitlength}{1mm} \psset{unit=1mm}
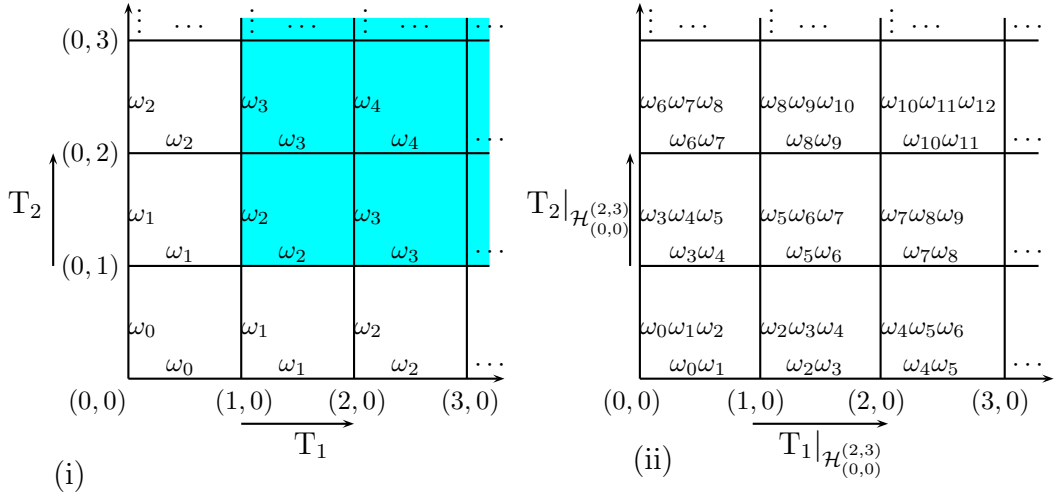
\begin{figure}[th]
\begin{center}
\begin{picture}(135,70)

\pspolygon*[linecolor=lightblue](30,35)(63,35)(63,68)(30,68)(30,35)

\psline{->}(15,20)(65,20)
\psline(15,35)(63,35)
\psline(15,50)(63,50)
\psline(15,65)(63,65)
\psline{->}(15,20)(15,70)
\psline(30,20)(30,68)
\psline(45,20)(45,68)
\psline(60,20)(60,68)

\put(7,16){\footnotesize{$(0,0)$}}
\put(26.5,16){\footnotesize{$(1,0)$}}
\put(41.5,16){\footnotesize{$(2,0)$}}
\put(56.5,16){\footnotesize{$(3,0)$}}

\put(20,21){\footnotesize{$\omega_{0}$}}
\put(35,21){\footnotesize{$\omega_{1}$}}
\put(50,21){\footnotesize{$\omega_{2}$}}
\put(61,21){\footnotesize{$\cdots$}}

\put(20,36){\footnotesize{$\omega_{1}$}}
\put(35,36){\footnotesize{$\omega_{2}$}}
\put(50,36){\footnotesize{$\omega_{3}$}}
\put(61,36){\footnotesize{$\cdots$}}

\put(20,51){\footnotesize{$\omega_{2}$}}
\put(35,51){\footnotesize{$\omega_{3}$}}
\put(50,51){\footnotesize{$\omega_{4}$}}
\put(61,51){\footnotesize{$\cdots$}}

\put(21,66){\footnotesize{$\cdots$}}
\put(36,66){\footnotesize{$\cdots$}}
\put(51,66){\footnotesize{$\cdots$}}
\put(61,66){\footnotesize{$\cdots$}}

\psline{->}(30,14)(45,14)
\put(37,10){$\rm{T}_1$}
\psline{->}(5,35)(5,50)
\put(-1,42){$\rm{T}_2$}

\put(6,34){\footnotesize{$(0,1)$}}
\put(6,49){\footnotesize{$(0,2)$}}
\put(6,64){\footnotesize{$(0,3)$}}

\put(15,26){\footnotesize{$\omega_{0}$}}
\put(15,41){\footnotesize{$\omega_{1}$}}
\put(15,56){\footnotesize{$\omega_{2}$}}
\put(16,66){\footnotesize{$\vdots$}}

\put(30,26){\footnotesize{$\omega_{1}$}}
\put(30,41){\footnotesize{$\omega_{2}$}}
\put(30,56){\footnotesize{$\omega_{3}$}}
\put(31,66){\footnotesize{$\vdots$}}

\put(45,26){\footnotesize{$\omega_{2}$}}
\put(45,41){\footnotesize{$\omega_{3}$}}
\put(45,56){\footnotesize{$\omega_{4}$}}
\put(46,66){\footnotesize{$\vdots$}}

\put(5,6){(i)}


\put(82,8){(ii)}

\psline{->}(98,14)(116,14)
\put(101,10){$\left.\rm{T}_{1}\right|_{\mathcal{H}_{(0,0)}^{(2,3)}}$}
\psline{->}(81.7,35)(81.7,50)
\put(67.5,42){$\left.\rm{T}_{2}\right|_{\mathcal{H}_{(0,0)}^{(2,3)}}$}

\psline{->}(83,20)(138,20)
\psline(83,35)(136,35)
\psline(83,50)(136,50)
\psline(83,65)(136,65)

\psline{->}(83,20)(83,70)
\psline(99,20)(99,68)
\psline(115,20)(115,68)
\psline(131.5,20)(131.5,68)

\put(78,16){\footnotesize{$(0,0)$}}
\put(94.5,16){\footnotesize{$(1,0)$}}
\put(110.5,16){\footnotesize{$(2,0)$}}
\put(127,16){\footnotesize{$(3,0)$}}

\put(87,21){\footnotesize{$\omega_{0}\omega_{1}$}}
\put(102.5,21){\footnotesize{$\omega_{2}\omega_{3}$}}
\put(118,21){\footnotesize{$\omega_{4}\omega_{5}$}}
\put(132.5,21){\footnotesize{$\cdots$}}

\put(87,36){\footnotesize{$\omega_{3}\omega_{4}$}}
\put(102.5,36){\footnotesize{$\omega_{5}\omega_{6}$}}
\put(118,36){\footnotesize{$\omega_{7}\omega_{8}$}}
\put(132.5,36){\footnotesize{$\cdots$}}

\put(87,51){\footnotesize{$\omega_{6}\omega_{7}$}}
\put(102.5,51){\footnotesize{$\omega_{8}\omega_{9}$}}
\put(118,51){\footnotesize{$\omega_{10}\omega_{11}$}}
\put(132.5,51){\footnotesize{$\cdots$}}

\put(88,66){\footnotesize{$\cdots$}}
\put(104,66){\footnotesize{$\cdots$}}
\put(119,66){\footnotesize{$\cdots$}}
\put(132.5,66){\footnotesize{$\cdots$}}

\put(83,26){\footnotesize{$\omega_{0}\omega_{1}\omega_{2}$}}
\put(83,41){\footnotesize{$\omega_{3}\omega_{4}\omega_{5}$}}
\put(83,56){\footnotesize{$\omega_{6}\omega_{7}\omega_{8}$}}
\put(84,66){\footnotesize{$\vdots$}}

\put(99,26){\footnotesize{$\omega_{2}\omega_{3}\omega_{4}$}}
\put(99,41){\footnotesize{$\omega_{5}\omega_{6}\omega_{7}$}}
\put(99,56){\footnotesize{$\omega_{8}\omega_{9}\omega_{10}$}}
\put(100,66){\footnotesize{$\vdots$}}

\put(115,26){\footnotesize{$\omega_{4}\omega_{5}\omega_{6}$}}
\put(115,41){\footnotesize{$\omega_{7}\omega_{8}\omega_{9}$}}
\put(115,56){\footnotesize{$\omega_{10}\omega_{11}\omega_{12}$}}
\put(116,66){\footnotesize{$\vdots$}}
\end{picture}
\end{center}
\caption{(i) Weight diagram of the $2$-variable weighted shift $W_{\theta(\omega)}$, highlighting its restriction to $\mathcal{M}\cap \mathcal{N}$ (see page \pageref{MN} for the definitions of $\mathcal{M}$ and $\mathcal{N}$); (ii) weight diagram used in Example \ref {ex102}.}
\label{Fig 3}
\end{figure}


\begin{lemma}
\label{Lem 2} \cite{CLY2} (a) \ Let $W_{\omega }$ be a unilateral weighted
shift, let $W_{\theta(\omega)}$ be the classical embedding of $\omega $, and
let $k\geq 1$. \ Then
\begin{equation}
\begin{tabular}{l}
$W_{\omega }$ is $k$-hyponormal if and only if $W_{\theta(\omega)}$ is $k$-hyponormal.
\end{tabular}
\label{k-hy}
\end{equation}%
(b) \ $W_{\omega }$ is subnormal if and only if $W_{\theta(\omega)}$ is subnormal; in this case, the Berger measure $%
\epsilon $ of $W_{\theta(\omega)}$ is diagonal. \ Moreover, $\epsilon ^{X}$ is the Berger measure of $W_{\omega }$.
\end{lemma}

\begin{example}
\label{ex101} Let $W_{\omega }$, $\sigma $, $W_{\left( \alpha(\omega) ,\beta(\omega) \right)}$ and $\gamma _{(k_{1},k_{2})}$ be as in (\ref{moment1}). \ If $p(r)=r$ and
$q(r)=r$, then the $(p,q)$-embedding $W_{\left( \alpha(\omega) ,\beta(\omega) \right) }$ of $W_{\omega }$ is the embedding considered in \cite{CLY2}; that is, $W_{\left( \alpha(\omega) ,\beta(\omega) \right)}=W_{\theta(\omega)}$, or equivalently,
$$
\beta_{(k_{1},k_{2})}=\alpha _{(k_{1},k_{2})}=\omega _{k_{1}+k_{2}} \quad (\textrm{for all } k_{1},k_{2}=0,1,2,\cdots)
$$
(see Figure \ref{Fig 3}(i)). \ To prove this, observe that
$$
\gamma _{(k_{1},k_{2})}\left( W_{\left( \alpha(\omega) ,\beta(\omega) \right)}\right)=\int r^{k_{1}}r^{k_{2}}\;d\sigma
(r)=\int r^{k_{1}+k_{2}}\;d\sigma (r)=\gamma _{k_{1}+k_{2}}\left(
W_{\omega }\right),
$$
and therefore $\beta _{(k_{1},k_{2})}=\alpha
_{(k_{1},k_{2})}=\omega _{k_{1}+k_{2}} \; (\textrm{all } k_{1},k_{2}=0,1,2,\cdots)$, as desired. \qed
\end{example}

\bigskip
\noindent
{\bf The Spherically Quasinormal Embedding}. \ 

We now state and prove that, starting with a non-flat subnormal shift $W_{\omega}$, the construction described in Subsection \ref{constr} always leads to a spherically quasinormal embedding.

\begin{theorem} \label{thm51}
(cf. Problem \ref{prob1}(i), and Problems \ref{quest} and \ref{Q1}) \ Let $W_{\omega}$ be a subnormal unilateral weighted shift, with Berger measure $\sigma$, and let $\mu:=\sigma \circ f^{-1}$, where $f(r):=(r,c-r)$, where $c:=\left\|W_{\omega}\right\|^2$. \ Then the $2$-variable weighted shift $\mathcal{PE}(\omega;f)$ is a spherically quasinormal embedding of $W_{\omega}$. 
\end{theorem}

\begin{proof}
We know that $\mathcal{PE}(\omega;f)$ is subnormal, and that $\supp \; \mu \subseteq \{(r,c-r):r \in [0,c]\}$. \ Then 
\begin{eqnarray*}
\gamma_{(k_1+1,k_2)}[\mu]+\gamma_{(k_1,k_2+1)}[\mu]&=&\int\int s^{k_1+1}t^{k_2} d\mu(s,t)+\int\int s^{k_1}t^{k_2+1} d\mu(s,t) \\
&=&\int \int (s+t)s^{k_1}t^{k_2} d\mu(s,t)=\int\int cs^{k_1}t^{k_2} d\mu(s,t) \\
&=&c\gamma_{(k_1,k_2)}[\mu],
\end{eqnarray*}
for all $k_1,k_2 \in \mathbb{Z}_+$. \ By Theorem \ref{Quasinormal}(iv), we see that $\mathcal{PE}(\omega;f)$ is spherically quasinormal, as desired.
\end{proof}

\newpage
\noindent
$(p,q)${\bf -embeddings of Recursively Generated Weighted Shifts}. \ 
\begin{theorem}
\label{Prop1A}(cf. Problem \ref{prob1}(ii)) \ Let $W_{\left( \alpha(\omega) ,\beta(\omega) \right) }$ be a 
$(p,q)$-embedding $2$-variable weighted shift, where $W_{\omega }$ is
subnormal with finitely atomic Berger measure $\sigma =\sum_{i=0}^{\ell
-1}\rho _{i}\delta _{r_{i}}$, $\sum_{i=0}^{\ell -1}\rho _{i}=1$, $\rho _{i}>0
$, and $0\leq r_0<r_1< \cdots <r_{\ell-1}\leq 1$. \ Then the Berger measure $\mu $ of $%
W_{\left( \alpha(\omega) ,\beta(\omega) \right) }$ is%
\begin{equation}
\mu =\sum_{i=0}^{\ell -1}\rho _{i}\delta _{\left(p(r_{i}),q(r_{i})\right)
}.  \label{finite measureA}
\end{equation}
\end{theorem}

\begin{proof}
A moment's thought reveals that, without loss of generality, we can assume that $p(r) \equiv r^m$ and $q(r) \equiv r^n$, where $m$ and $n$ are positive integers. \ We first observe that $\supp \; \mu \subseteq \left\{ \left(
p(r),q(r)\right) :0\leq r\leq 1\right\} $. \ Since
\begin{eqnarray*}
\int \int s^{k_{1}}t^{k_{2}}\left( t^{m}-s^{n}\right) d\mu \left(
s,t\right)& =&\int \int \left(
s^{k_{1}}t^{k_{2}+m}-s^{k_{1}+n}t^{k_{2}}\right) d\mu \left( s,t\right) \\
&=&\int_{\left[ 0,1\right] }\left( r^{mk_{1}}r^{n\left( k_{2}+m\right)
}-r^{m\left( k_{1}+n\right) }r^{nk_{2}}\right) d\sigma \left( r\right) =0,
\end{eqnarray*}
for all $k_1,k_2 \ge 0$, we easily see that $t^{m}=s^{n}$ a.e. $\mu$. \ Thus, $\supp \; \mu
\subseteq \left\{ \left( r^{m},r^{n}\right) :0\leq r\leq 1\right\} $.

The rest of the proof will follow once we establish the next two Claims.

\textbf{Claim 1:} If $\sigma =\delta _{s_{0}}$, then $\mu =\delta _{\left(
s_{0}^m,s_{0}^n\right) }$.

\textbf{Proof of Claim 1:} \ Observe that for all $(k_{1},k_{2})\in \mathbb{Z%
}_{+}^{2}$,
$$
\gamma _{mk_{1}+nk_{2}}\left( W_{\omega }\right) =\int_{\left[ 0,1\right]
}r^{mk_{1}+nk_{2}}d\sigma \left( r\right)
=\int_{\left[ 0,1\right]
} p(r)^{k_{1}}q(r)^{k_{2}}d\sigma (r)=\gamma _{(k_{1},k_{2})}\left(
W_{\left( \alpha(\omega) ,\beta(\omega) \right) } \right).
$$
It follows that
\begin{eqnarray*}
\int \int s^{k_{1}}t^{k_{2}}d\delta _{\left(
s_{0},s_{0}\right)}(s,t)
&=&s_{0}^{mk_{1}}s_{0}^{nk_{2}}=s_{0}^{mk_{1}+nk_{2}} \\
&=&\int r^{mk_1+nk_2}d\delta_{s_0}(r)=\gamma_{mk_1+nk_2}[\sigma] \\
&=&\gamma_{(k_1,k_2}(W_{(\alpha(\omega),\beta(\omega))}).
\end{eqnarray*}

Thus, we must have $\mu =\delta _{\left( s_{0}^m,s_{0}^n\right) }$.

\textbf{Claim 2:} If $\sigma =\sum_{i=0}^{\ell -1}\rho _{i}\delta _{s_{i}}$,
then $\mu =\sum_{i=0}^{\ell -1}\rho _{i}\delta _{\left(
s_{i}^{m},s_{i}^{n}\right) }$.

\textbf{Proof of Claim 2:} \  This is straightforward.

\end{proof}

\begin{corollary} \label{newcor}
Let $W_{(\alpha(\omega),\beta(\omega))}$ be a polynomial embedding of $W_{\omega}$, and let $\mu$, $\sigma$ and $f:=(p,q)$ be as in Theorem \ref{Th4A}; that is,
$$
d \mu(p(r),q(r))=d \sigma(r).
$$
Then
$$
d \left( \mu(p(r),q(r)) \right)^X=d \sigma(p(r))
$$
and
$$
d \left( \mu(p(r),q(r)) \right)^Y=d \sigma(q(r))
$$
where for a measure $\nu$, $(\nu)^X$ and $(\nu)^Y$ denote the marginal measures of $\nu$. \ Alternatively, recall that 
$\mu:=\sigma \circ f^{-1}$. \ Then
\begin{equation} \label{marginalX}
\mu^X=\sigma \circ p^{-1}
\end{equation}
and
$$
\mu^Y=\sigma \circ q^{-1}.
$$
As a consequence, 
$$
\mu^X \circ p=\sigma
$$
and
$$
\mu^Y \circ q=\sigma,
$$
so that
$$
\mu^X \circ p=\mu^Y \circ q.
$$
In the case when $\sigma$ is the Lebesgue measure on $[0,1]$, we see that 
$$
d\mu^X(s)=d(p^{-1}(s))=(p^{-1})'(s)ds.
$$
and
$$
d\mu^Y(t)=d(q^{-1}(t))=(q^{-1})'(t)dt.
$$
(For additional details on marginal measures, the reader is referred to Subsection \ref{disintegration} in the Appendix.) 
\end{corollary}

\begin{proof}
We focus on the marginal measure $\mu^X$. \ From Subsection \ref{disintegration}, we know that 
$$
\mu^X=\mu \circ \pi_X^{-1}.
$$
Therefore,
$$
d \mu^X(p(r))=d\mu(p(r),q(r))=d\sigma(r)
$$
(where we have used the injectivity of $f$ to claim that there is a unique point $(p(r),q(r))$ projecting onto $p(r)$). \ Then 
$$
\mu^X \circ p = \sigma,
$$
or
$$
\mu^X=\sigma \circ p^{-1},
$$
as desired. 

For the reader's convenience, we also provide an alternative proof, that uses techniques from convex analysis. \ First, observe that the result is obvious if $\sigma$ is a point mass $\delta_{r}$ in (\ref{marginalX}). \ As a consequence, the result is also true for finitely atomic measures. \ Moreover, we know that these point masses are the extreme points of the unit ball of the space of probability Borel measures on either a closed interval in $\mathbb{R}$ or a closed rectangle in $\mathbb{R}^2$. \ Next, each such probability Borel measure is the $w^*$-limit of a net of finitely atomic measures. \ Since the map that sends $\mu$ to its marginal measures $\mu^X$ and $\mu^Y$ is $w^*$-continuous, it follows that Corollary \ref{newcor} now holds for all probability Borel measures.
\end{proof}

It is well known that, for continuous functions on a closed interval, the Lebesgue integral can be calculated as a limit of Riemann sums over partitions whose norms tend to zero \cite{Pihl}. \ We shall use this fact to find explicitly the Berger measure $\mu$ of a $(p,q)$-embedding from the Berger measure $\sigma$ of the original unilateral weighted shift, when the latter is the Bergman shift.
 \ Concretely, we prove that the Berger measure of every $(p,q)$-embedding of $B_+$ is normalized arclength on the graph of $(p,q)$.
 
\begin{theorem} \label{arclengthA} 
(cf. Problem \ref{quest}) \ Let $W_{(\alpha(\omega),\beta(\omega))}$ be a $(p,q)$-embedding of a subnormal weighted shift $W_{\omega}$, and denote by $\mu$ and $\sigma$ the Berger measures of $W_{(\alpha(\omega),\beta(\omega))}$ and $W_{\omega}$, respectively. \ Assume that $\sigma$ is the Lebesgue measure on $[0,1]$. \ Then $d\mu(p(r),q(r))$ is normalized arclength on the curve $C:=\{(p(r),q(r)):0 \le r \le 1\}$.  
\end{theorem}

\begin{proof}
Consider a partition $\mathcal{P} \equiv \{0 \le r_0<r_1<\cdots<r_{N-1}\}$, 
and let $\left\|\mathcal{P}\right\|$ denote the norm of $\mathcal{P}$. \ 
The net $\{\sum_{i=0}^{N-1}(r_{i+1}-r_i)\delta_{\widetilde{r_i}}\}_{\mathcal{P}}$, 
running over the partially ordered set of all partitions $\mathcal{P}$, 
converges in the $w^*$-topology to the Lebesgue measure on $[0,1]$. \ (As usual, 
the points $\widetilde{r_i}$ belong to the interval $[r_i,r_{i+1}]$.) \ 
For a continuous function $g$ on $[0,1]$, a typical Riemann sum is of the form $\sum_{i=0}^{N-1}(r_{i+1}-r_i)g(\widetilde{r_i})$; thus, it can be interpreted as integration with respect to the $2$-variable finitely atomic measure 
\begin{equation} \label{Riemann}
\sum_{i=0}^{N-1}(r_{i+1}-r_i)\delta_{(p(\widetilde{r_i}),q(\widetilde{r_i}))},
\end{equation}
given by Theorem \ref{Prop1}. \ We also know that the support of $\mu$ is contained in the curve $C$. \ Thus, the set 
$\mathcal{P}(C):=\{(p(\widetilde{r_0}),q(\widetilde{r_0})), \cdots, (p(\widetilde{r_{N-1}}),q(\widetilde{r_{N-1}})\}$ is a partition in the curve $C$, which can be used to approximate the Lebesgue measure on $C$. \ 
Moreover, since the Riemann sum over the original partition $\mathcal{P}$ (thought of as an integral with respect to a finitely atomic measure) has total mass equal to $1$, the finitely atomic measure associated to the partition $\mathcal{P}(C)$ is a probability measure. \ 

It follows that the finitely atomic measures involved in (\ref{Riemann}) are probability measures, as will be any $w^*$-limit of that net. \ If we could establish that the norm of $\mathcal{P}(C)$, $\left\|\mathcal{P}(C)\right\|$ is majorized by a constant multiple of $\left\|\mathcal{P}\right\|$, it would then follow that the Riemann sums in (\ref{Riemann}) $w^*$-converge to normalized Lebesgue measure on $C$; that is, it would $w^*$-converge to normalized arclength. \ 

To complete the proof, we observe that the distance between two consecutive points in $\mathcal{P}(C)$ is at most $\sqrt{2}\max \{\left\|p'\right\|_{\infty},\left\|q'\right\|_{\infty}\}\left\|\mathcal{P}\right\|$, and this quantity majorizes 
$\left\|\mathcal{P}(C)\right\|$, where $p'$ and $q'$ denote the derivatives of $p$ and $q$, respectively. \ (Observe that the norm $\left\|\cdot\right\|_{\infty}$ is calculated over the interval $[0,1]$.)
\end{proof}

\begin{example}
On the closed interval $[0,1]$, consider the $(p,q)$-embedding $W_{(\alpha(\omega),\beta(\omega))}$ generated by the nonnegative polynomials $p(r):=r$ and $q(r):=r(1-r)$. \ By Theorem \ref{Th4A}, the Berger measure $\mu$ of the embedding is related to the Berger measure $\sigma$ of $W_{\omega}$ by the identity
$$
d\mu(r,r(1-r))=d\sigma(r) \quad (r \in [0,1]).
$$
It follows that the moments of $\mu$ and $\sigma$ are related by the identity
\begin{eqnarray} \label{inteq}
\gamma_{(k,\ell)}[\mu]&=&\int_{[0,1]}r^k(r(1-r))^{\ell}d\sigma(r)=(-1)^{\ell}\int_{[0,1]}r^{k+\ell}(r-1)^{\ell}d\sigma(r) \nonumber \\
&=&(-1)^{\ell}\int r^{k+\ell}\sum_{i=0}^{\ell}(-1)^{\ell-i} {\genfrac(){0pt}{0}{\ell}{i}} r^i d\sigma(r)=\int \sum_{i=0}^{\ell}(-1)^i {\genfrac(){0pt}{0}{\ell}{i}} r^{k+\ell+i}d\sigma(r) \nonumber \\
&=&\sum_{i=0}^{\ell}(-1)^i {\genfrac(){0pt}{0}{\ell}{i}} \gamma_{k+\ell+i}[\sigma].
\end{eqnarray}
If we now let $\sigma \equiv \lambda$, the Lebesgue measure on $[0,1]$, we know by Theorem \ref{arclengthA} that $\mu$ is normalized arclength on the curve $t=s-s^2$, for $0 \le s \le 1$. \ The moments of $\mu$ are 
\begin{eqnarray*}
\gamma_{(k,\ell)}[\mu]&=&\int_{[0,1]}r^{k+\ell}(1-r)^{\ell}dr=B(k+\ell+1,\ell+1)
=\frac{\Gamma(k+\ell+1)\Gamma(\ell+1)}{\Gamma(k+2\ell+2)} \\
&=&\frac{(k+\ell)!\ell !}{(k+2\ell+1)!},
\end{eqnarray*}
where $B$ and $\Gamma$ denote the classical Beta and Gamma functions, respectively. \ On the other hand, (\ref{inteq}) says that
$$
\gamma_{(k,\ell)}[\mu]=\sum_{i=0}^{\ell}(-1)^i {\genfrac(){0pt}{0}{\ell}{i}} \gamma_{k+\ell+i}[\lambda]=
\sum_{i=0}^{\ell}(-1)^i {\genfrac(){0pt}{0}{\ell}{i}} \frac{1}{k+\ell+i+1},
$$
and therefore
\begin{equation} \label{newequation}
\sum_{i=0}^{\ell}(-1)^i {\genfrac(){0pt}{0}{\ell}{i}} \frac{1}{k+\ell+1+i}=\frac{(k+\ell)!\ell !}{(k+2\ell+1)!}.
\end{equation}
In particular, 
$$
\lim_{\ell \rightarrow \infty}\sum_{i=0}^{\ell}(-1)^i {\genfrac(){0pt}{0}{\ell}{i}} \frac{1}{k+\ell+1+i}=0.
$$
While (\ref{newequation}) is known (see for instance \cite[vol. 3, (3.8)]{Gould}), it is interesting that it can be obtained as a by-product of a calculation of moments for the $(r,r(1-r))$-embedding.

\end{example}

\bigskip
{\bf The Neil Parabolic Embedding}. \ The {\it Neil parabola} 
$$
NP:=\{(r^2,r^3):r \in [0,1]\}
$$
(also known as the {\it semicubical parabola}) has been studied in various contexts; see for instance \cite{Kne} and \cite{Pic}. \ 
Here we will use it to produce a simple example, which will lead to an answer to Problem \ref{prob5}. \ First, we recall that given a measure $\mu $ on $X\times Y$, the \textit{marginal measure} $\mu
^{X}$ is given by $\mu ^{X}:=\mu \circ \pi _{X}^{-1}$, where $\pi
_{X}:X\times Y\rightarrow X$ is the canonical projection onto $X$. \ Thus $%
\mu ^{X}(E)=\mu (E\times Y)$, for every $E\subseteq X$ (cf. Subsection \ref{disintegration}). 
 
\begin{example}
\label{ex102} Let $W_{\omega }$ be a subnormal unilateral weighted shift, with Berger measure $\sigma$, let $p(r)=r^{2}$ and $q(r)=r^{3}$, and let $f:=(p,q)$. \ Then 
the $(p,q)$-embedding $\mathcal{PE}(\omega;f)$ is subnormal, with Berger measure $\mu$ supported on $NP$, such that
$$
d\mu(r^2,r^3)=d\sigma(r).
$$
Furthermore, $W_{\left( \alpha ,\beta \right) }$ has the weight diagram
shown in Figure \ref{Fig 3}(ii). \ By direct calculation, one observes that the weight diagram for this embedding is intimately related to the weight diagram for the integer powers of $W_{\theta(\omega)}$. \ Since $W_{\theta(\omega)}$ is subnormal, let us denote by 
$\nu$ its Berger measure; we know that $\supp \; \nu \subseteq \{(u,u): u \in \mathbb{R}_+ \}$. \ For $(s,t) \in NP$, we then have 
$$
d\mu \left( s,t\right) =d\nu \left( s^{\frac{1}{2}},t^{\frac{1}{3}}\right).
$$
Moreover, the last expression also represents the Berger measure of $W_{\theta(\omega)}$ restricted to the reducing subspace 
$\mathcal{H}_{(0,0)}^{(2,3)}$ (see (\ref{decomposition}) in Section \ref{sectionclassical} for the precise definition of this subspace). \ In the special case of $W_{\omega }=B_{+}$, we know that $\sigma$ is Lebesgue measure on $[0,1]$, and therefore
$$
d\mu(r^2,r^3)=dr.
$$
By Corollary \ref{newcor}, 
$$
d\left(\mu^X \left( s\right)\right) =d\left( s^{\frac{1}{2}}\right) =\frac{1}{2\sqrt{s}}ds
$$
and
$$
d\left(\mu^Y \left( t\right)\right) =d\left( t^{\frac{1}{3}}\right) =\frac{1}{3t^{\frac{2}{3}}}dt.
$$
A simple calculation of moments reveals that this consistent with the information in Figure \ref{Fig 3}(ii). \qed
\end{example}



\section{Application 1: The Classical Case -- $k$-hyponormality in the class $\{\mathcal{CE}(\omega):W_{\omega} \textrm{ a subnormal unilateral weighted shift} \}$ } \label{sectionclassical}

In this section, we consider the class of $2$-variable weighted shifts obtained as classical embeddings of unilateral weighted shifts. \ We ask whether a
power of $W_{\theta(\omega)}$ can be (jointly) $k$-hyponormal, under the assumption that $W_{\omega }$ is $k$-hyponormal $\left( k\geq 1\right) $. \ The following
examples will answer this question. \ First, recall that for every positive integer $m$ we have
\begin{equation} \label{orthosum}
W_{\omega }^{m}\cong W_{\omega \left( m:0\right) }\oplus W_{\omega \left(
m:1\right) }\oplus \cdots \oplus W_{\omega \left( m:m-1\right) },
\end{equation}
where the unilateral weighted shift $W_{\omega \left( m:i\right) }$ is associated with the sequence $\omega\left( m:i\right):=
\{\omega_{i+k}\omega_{i+k+1} \cdots \omega_{i+k+m-1}\}_{k \ge 0}$, for $i=0,1,\cdots,m-1$ (cf. \cite{CuP}).

If $W_{\omega }$ is subnormal with Berger measure $\sigma $, then each summand $W_{\omega \left( m:i\right) }$ is subnormal $(0\leq i\leq m-1)$. \ Let $\sigma _{i}$
be the Berger measure for $W_{\omega \left(
m:i\right) } \; (0 \le i \le m-1)$. \label{sigma0} \ In what follows, we will need the following result.

\begin{theorem}(\cite[Theorem 2.9]{CuP}) \label{CurtoPark}
$$
d\sigma _{0}\left( r\right) =d\sigma \left(
r^{\frac{1}{m}}\right)
$$
and 
$$
d\sigma _{i}\left( r\right) =\frac{r^{i}}{\omega
_{0}^{2}\cdots \omega _{i-1}^{2}}d\sigma \left( r^{\frac{1}{m}}\right) \; \; (1 \le i \le m-1).
$$
\end{theorem}

Related to the decomposition (\ref{orthosum}) and Theorem \ref{CurtoPark}, we now prove:

\begin{theorem}
\label{Thm1}Let $r$ be a nonnegative real number and $0\leq i\leq m-1$. \medskip
\newline
(i) if $r\in \supp \; \sigma $ and $r\neq 0$, then $r^{m}\in \supp \; \sigma _{i} \; \; (1 \le i \le m-1)$;\medskip \newline
(ii) if $r\in \supp \; \sigma _{i}$ and $r\neq 0$,
then $r^{\frac{1}{m}}\in \supp \; \sigma \; (1 \le i \le m-1)$;\medskip \newline
(iii) $0\in \supp \; \sigma $ if and only if $0\in \supp \; \sigma _{0}$.
\end{theorem}

\begin{proof}
Let $\mathrm{Bo}(\mathbb{R}_{+})$ denote the Borel $\sigma $-algebra on $%
\mathbb{R}_{+}$. \ For $N \in \mathrm{Bo}(\mathbb{R}_{+})$, we let 
$N^{\frac{1}{m}}:=\left\{ r^{\frac{1}{m}}:r\in N \right\}$ and 
$N^{m}:=\left\{ r^{m}:r\in N \right\}$.\medskip \newline
(i): \ Since $r\in \supp \; \sigma $, there exists $N_{r}\in \mathrm{Bo}(%
\mathbb{R}_{+})$ such that $r\in N_{r}$ and $\sigma \left( N_{r}\right) >0$. \
Note that
\begin{equation}
\sigma _{0}\left( N_{r}^{m}\right) >0\Longleftrightarrow \sigma \left(N_{r}\right) >0 \label{relation 11} 
\end{equation} 
and 
\begin{equation}
\sigma _{i}\left( N_{r}^{m}\right) >0\Longleftrightarrow \sigma \left( N_{r}\right) >0. \label{relation 12}
\end{equation} 
Since $N_{r}^{m}\in \mathrm{Bo}(\mathbb{R}_{+})$, by the definition of $N_{r}^{m}$ and (\ref{relation 12}) we must have $r^{m}\in \supp \; \sigma _{i}$,
as desired.\medskip \newline
(ii): \ Since $r\in \supp \; \sigma _{i}$, there exists $N_{r}\in \mathrm{Bo%
}(\mathbb{R}_{+})$ such that $r\in N_{r}$ and $\sigma _{i}\left( N_{r}\right)
>0 $. \ Note that
\begin{equation}
\sigma _{0}\left( N_{r}\right) >0\Longleftrightarrow \sigma \left( N_{r}^{\frac{1}{m}}\right) >0 \label{relation 21}
\end{equation} 
and 
\begin{equation}
\sigma _{i}\left( N_{r}\right) >0\Longleftrightarrow \sigma \left( N_{r}^{\frac{1}{m}}\right) >0. \label{relation 22}
\end{equation} 
Since $N_{r}^{\frac{1}{m}}\in \mathrm{Bo}(\mathbb{R}_{+})$, by the definition of $N_{r}^{\frac{1}{m}}$ and (\ref{relation 22}) we must have $r^{\frac{1}{m}}\in \supp \; \sigma $, as desired.\newline
(iii): \ Since $d\sigma _{0}\left( r\right) =d\sigma \left( r^{\frac{1}{m}}\right)
$, the equivalences (\ref{relation 11}) and (\ref{relation 21}) give the desired conclusion.
\end{proof}

We now establish:

\begin{theorem}
\label{Thm2}(cf. Problem \ref{prob2}(ii)) \ Let $W_{\omega }$ be a subnormal unilateral weighted shift, and let $m$ be a positive integer. \ Assume that $W_{\omega }^{m}$ is a direct sum of recursively generated weighted shifts. \ Then $W_{\omega }$ is a recursively generated weighted shift.
\end{theorem}

\begin{proof}
This follows easily from Theorem \ref{Thm1} and the fact that a subnormal unilateral weighted shift $W_{\omega}$ is recursively generated if and only if its Berger measure $\sigma$ is finitely atomic, that is, $\card (\supp \; \sigma) < +\infty$ (cf. Subsection \ref{RGWS} in the Appendix).
\end{proof}

By analogy with the decomposition used in \cite{CuP}, we now introduce a decomposition for a commuting $2$-variable weighted shift $W_{\left( \alpha ,\beta \right) }$ which will allow us to analyze its integer powers. \ Concretely, we split the ambient space 
$\mathcal{H} \equiv \ell ^{2}(\mathbb{Z}_{+}^{2})$ as an orthogonal direct sum $\mathcal{H} \equiv \bigoplus_{p=0}^{m-1}\bigoplus_{q=0}^{n-1}\mathcal{H}_{(p,q)}^{(m,n)}$, where
\begin{equation}
\mathcal{H}_{(p,q)}^{(m,n)}:=\vee \{e_{(mi+p,nj+q)}:i=0,1,2,\cdots
,j=0,1,2,\cdots \} \quad \quad
\left\{
\begin{tabular}{l}
$(0 \le p \le m-1)$ \\
$(0 \le q \le n-1)$
\end{tabular}
\right. \label{decomposition}
\end{equation} 
(see Figure \ref{Fig 5}).  



\setlength{\unitlength}{1mm} \psset{unit=1mm}
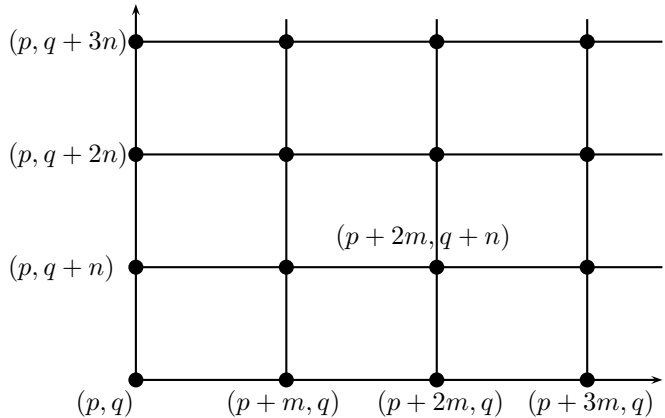
\begin{figure}[th]
\begin{center}
\begin{picture}(135,55)

\rput(18,25){
\psline{->}(15,5)(85,5)
\psline(15,20)(85,20)
\psline(15,35)(85,35)
\psline(15,50)(85,50)

\psline{->}(15,5)(15,55)
\psline(35,5)(35,53)
\psline(55,5)(55,53)
\psline(75,5)(75,53)

\put(7,1){\footnotesize{$(p,q)$}}
\put(27,1){\footnotesize{$(p+m,q)$}}
\put(47,1){\footnotesize{$(p+2m,q)$}}
\put(67,1){\footnotesize{$(p+3m,q)$}}

\put(-2,19){\footnotesize{$(p,q+n)$}}
\put(-2,34){\footnotesize{$(p,q+2n)$}}
\put(-2,49){\footnotesize{$(p,q+3n)$}}

\pscircle*(15,5){1}
\pscircle*(15,20){1}
\pscircle*(15,35){1}
\pscircle*(15,50){1}
\pscircle*(35,5){1}
\pscircle*(35,20){1}
\pscircle*(35,35){1}
\pscircle*(35,50){1}
\pscircle*(55,5){1}
\pscircle*(55,20){1}
\pscircle*(55,35){1}
\pscircle*(55,50){1}
\pscircle*(75,5){1}
\pscircle*(75,20){1}
\pscircle*(75,35){1}
\pscircle*(75,50){1}

\put(41.5,23){\footnotesize{$(p+2m,q+n)$}}
}
\end{picture}
\end{center}
\caption{Schematic diagram of the subspace $\mathcal{H}_{(p,q)}^{(m,n)}$}
\label{Fig 5}
\end{figure}


We now let $\mathcal{M}:=\vee\{e_{(k_1,k_2)}:k_1 \ge 0, k_2 \ge 1 \}$ and $\mathcal{N}:=\vee\{e_{(k_1,k_2)}:k_1 \ge 1, k_2 \ge 0 \}$. \label{MN} 

Each of the subspaces $\mathcal{H}_{(p,q)}^{(m,n)}$ reduces $T_{1}^{m}$ and $T_{2}^{n}$,
and $W_{\left( \alpha ,\beta \right) }^{(m,n)}$ is subnormal if and only if
each $\left.W_{\left( \alpha ,\beta \right) }^{(m,n)}\right|_{\mathcal{H}%
_{(p,q)}^{(m,n)}}$ is subnormal. \ Moreover, each $\left.W_{\left( \alpha ,\beta \right) }^{(m,n)}\right|_{\mathcal{H}%
_{(p,q)}^{(m,n)}}$ is a $2$-variable weighted shift.  

We now recall a simple criterion to detect the $k$-hyponormality of $2$%
-variable weighted shifts.

\begin{lemma}
\label{khypo} (\cite[Theorem 2.4]{CLY1}) The following statements are
equivalent:\newline
(i) $\ W_{(\alpha ,\beta )}$ is $k$-hyponormal;\newline
(ii) $(\gamma
_{(u_{1},u_{2})+(n,m)+(p,q)})_{_{0\leq p+q\leq k}^{0\leq n+m\leq k}}\geq 0$
for all $(u_{1},u_{2})\in \mathbb{Z}_{+}^{2}$.
\end{lemma}

\begin{example}
\label{ex103} Consider $W_{\theta(\omega)} \equiv \mathbf{(}%
T_{1},T_{2})$ given by Figure \ref{Fig 1}(ii), that is, $\alpha
_{(k_{1},k_{2})}=\beta _{(k_{1},k_{2})}=\gamma_{k_1+k_2}(\omega)$ for all $k_{1},k_{2}\ge0$. \ Assume that 
\begin{equation*}
W_{\omega }=\shift (\alpha _{(0,0)},\alpha _{(1,0)},\cdots )=\shift (\sqrt{x},\sqrt{\frac{2}{3}},\sqrt{\frac{3}{4}},\sqrt{\frac{4}{5}}%
,\cdots ).
\end{equation*}
($W_{\omega}$ is a rank-one perturbation of $B_+$, and it was previously studied in \cite{QHWS}.) \ Then, we have:\newline
(i) \ $W_{\omega }$ is hyponormal if and only if $W_{\theta(\omega)}$ is (jointly) hyponormal if and
only if $0<x\leq \frac{2}{3}\simeq0.667$.\medskip \newline
(ii) \ $W_{\omega }$ is $2$-hyponormal if and only if $W_{\theta(\omega)}$ is (jointly) $2$-hyponormal if and only if $0<x\leq
\frac{9}{16}\simeq 0.563$.$\medskip $\newline
(iii) \ $W_{\omega }$ is $3$-hyponormal if and only if $W_{\theta(\omega)}$ is (jointly) $3$-hyponormal if and only if $0<x\leq
\frac{8}{15}\simeq 0.533$.$\medskip $\newline
(iv) \ $W_{\left( \alpha ,\beta \right) }^{\left( 2,3\right) }=W_{\theta(\omega)}^{\left( 2,3\right) }$ is (jointly) hyponormal if and
only if $0<x\leq \frac{2}{3}\simeq 0.667$.$\medskip $\newline
(v) \ $\left.W_{\theta(\omega)}^{\left( 2,3\right) }\right|_{\mathcal{H}_{(0,0)}^{(2,3)}}$ is (jointly) $2$-hyponormal if and only if $0<x\leq \frac{%
49}{90}\simeq 0.544$.$\medskip $\newline
As a consequence, we have:\newline
(vi) \ $W_{\theta(\omega)}$ is (jointly) $2$-hyponormal but $W_{\theta(\omega)}^{\left( 2,3\right) }$ is not (jointly)
2-hyponormal, when $\frac{49}{90}<x\leq \frac{9}{16}$.

\medskip
We establish (i)--(vi) by direct calculations using (\ref{decomposition}), Lemma \ref{khypo}, and {\it Mathematica} \cite{Wol}. \qed
\end{example}



\setlength{\unitlength}{1mm} \psset{unit=1mm}
\begin{figure}[th]
\begin{center}
\begin{picture}(135,70)

\psline{->}(20,20)(70,20)
\psline(20,35)(68,35)
\psline(20,50)(68,50)
\psline(20,65)(68,65)
\psline{->}(20,20)(20,70)
\psline(35,20)(35,68)
\psline(50,20)(50,68)
\psline(65,20)(65,68)

\put(12,16){\footnotesize{$(0,0)$}}
\put(31.5,16){\footnotesize{$(1,0)$}}
\put(46.5,16){\footnotesize{$(2,0)$}}
\put(61.5,16){\footnotesize{$(3,0)$}}

\put(25,21){\footnotesize{$\alpha_{(0,0)}$}}
\put(40,21){\footnotesize{$\alpha_{(1,0)}$}}
\put(55,21){\footnotesize{$\alpha_{(2,0)}$}}
\put(66,21){\footnotesize{$\cdots$}}

\put(25,36){\footnotesize{$\alpha_{(0,1)}$}}
\put(40,36){\footnotesize{$\alpha_{(1,1)}$}}
\put(55,36){\footnotesize{$\alpha_{(2,1)}$}}
\put(66,36){\footnotesize{$\cdots$}}

\put(25,51){\footnotesize{$\alpha_{(0,2)}$}}
\put(40,51){\footnotesize{$\alpha_{(1,2)}$}}
\put(55,51){\footnotesize{$\alpha_{(2,2)}$}}
\put(66,51){\footnotesize{$\cdots$}}

\put(26,66){\footnotesize{$\cdots$}}
\put(41,66){\footnotesize{$\cdots$}}
\put(56,66){\footnotesize{$\cdots$}}
\put(66,66){\footnotesize{$\cdots$}}

\psline{->}(35,14)(50,14)
\put(42,10){$\rm{T}_1$}
\psline{->}(10,35)(10,50)
\put(4,42){$\rm{T}_2$}

\put(11,34){\footnotesize{$(0,1)$}}
\put(11,49){\footnotesize{$(0,2)$}}
\put(11,64){\footnotesize{$(0,3)$}}

\put(20,26){\footnotesize{$\beta_{(0,0)}$}}
\put(20,41){\footnotesize{$\beta_{(0,1)}$}}
\put(20,56){\footnotesize{$\beta_{(0,2)}$}}
\put(21,66){\footnotesize{$\vdots$}}

\put(35,26){\footnotesize{$\beta_{(1,0)}$}}
\put(35,41){\footnotesize{$\beta_{(1,1)}$}}
\put(35,56){\footnotesize{$\beta_{(1,2)}$}}
\put(36,66){\footnotesize{$\vdots$}}

\put(50,26){\footnotesize{$\beta_{(2,0)}$}}
\put(50,41){\footnotesize{$\beta_{(2,1)}$}}
\put(50,56){\footnotesize{$\beta_{(2,2)}$}}
\put(51,66){\footnotesize{$\vdots$}}

\put(10,6){(i)}


\put(85,8){(ii)}

\psline{->}(95,14)(110,14)
\put(102,10){$\rm{T}_1$}
\psline{->}(77,35)(77,50)
\put(72,42){$\rm{T}_2$}

\psline{->}(80,20)(130,20)
\psline(80,35)(128,35)
\psline(80,50)(128,50)
\psline(80,65)(128,65)

\psline{->}(80,20)(80,70)
\psline(95,20)(95,68)
\psline(110,20)(110,68)
\psline(125,20)(125,68)

\put(75,16){\footnotesize{$(0,0)$}}
\put(91,16){\footnotesize{$(1,0)$}}
\put(106,16){\footnotesize{$(2,0)$}}
\put(121,16){\footnotesize{$(3,0)$}}

\put(85,21){\footnotesize{$\omega_{0}$}}
\put(100,21){\footnotesize{$\omega_{1}$}}
\put(115,21){\footnotesize{$\omega_{2}$}}
\put(126,21){\footnotesize{$\cdots$}}

\put(85,36){\footnotesize{$\omega_{1}$}}
\put(100,36){\footnotesize{$\omega_{2}$}}
\put(115,36){\footnotesize{$\omega_{3}$}}
\put(126,36){\footnotesize{$\cdots$}}

\put(85,51){\footnotesize{$\omega_{2}$}}
\put(100,51){\footnotesize{$\omega_{3}$}}
\put(115,51){\footnotesize{$\omega_{4}$}}
\put(126,51){\footnotesize{$\cdots$}}

\put(85,66){\footnotesize{$\cdots$}}
\put(100,66){\footnotesize{$\cdots$}}
\put(115,66){\footnotesize{$\cdots$}}
\put(126,66){\footnotesize{$\cdots$}}

\put(80,26){\footnotesize{$\omega_{0}$}}
\put(80,41){\footnotesize{$\omega_{1}$}}
\put(80,56){\footnotesize{$\omega_{2}$}}
\put(81,66){\footnotesize{$\vdots$}}

\put(95,26){\footnotesize{$\omega_{1}$}}
\put(95,41){\footnotesize{$\omega_{2}$}}
\put(95,56){\footnotesize{$\omega_{3}$}}
\put(96,66){\footnotesize{$\vdots$}}

\put(110,26){\footnotesize{$\omega_{2}$}}
\put(110,41){\footnotesize{$\omega_{3}$}}
\put(110,56){\footnotesize{$\omega_{4}$}}
\put(111,66){\footnotesize{$\vdots$}}
\end{picture}
\end{center}
\caption{(i) \ Weight diagram of a generic $2$-variable weighted shift; (ii) \ weight
diagram of a generic classical embedding $W_{\theta(\omega)} \equiv \mathbf{(}T_{1},T_{2})$.}
\label{Fig 1}
\end{figure}
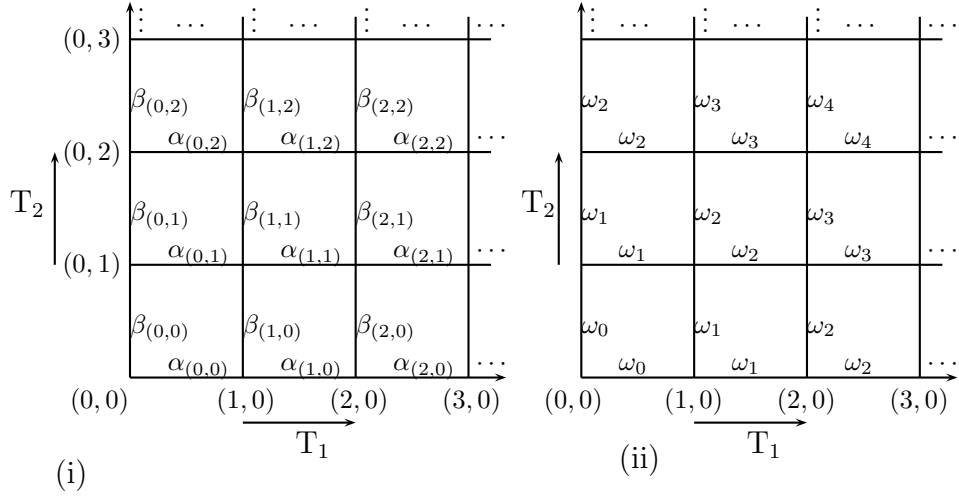


\begin{example}
\label{ex104} For $x \in [\frac{1}{2},\frac{2}{3}]$, let $\omega(x):=\shift(\sqrt{\frac{1}{2}},\sqrt{\frac{1}{2}},\sqrt{\frac{1}{2}},
\sqrt{x},\sqrt{\frac{2}{3}},\sqrt{\frac{3}{4}},\sqrt{\frac{4}{5}},\cdots )$. Then
\begin{itemize}
\item[(i)] \ $W_{\theta(\omega(x))}$ is hyponormal, but $W_{\theta(\omega(x))}^{\left( 2,3\right) }$ is not hyponormal.
\item[(ii)] \ $W_{\theta(\omega(x))}$ is not $2$-hyponormal.
\item[(iii)] \ $\left.W_{\theta(\omega(x))} \right|_{\mathcal{M}\cap \mathcal{N}}$
and $[\left.W_{\theta(\omega(x))} \right|_{\mathcal{M}\cap \mathcal{N}}]^{\left( 2,2\right) }$ are {\it not} $2$-hyponormal. \newline 
\item[(iv)] \ $[\left.W_{\theta(\omega(x))} \right|_{\mathcal{M}\cap \mathcal{N}}]^{\left( 3,3\right) }$ and $[\left.W_{\theta(\omega(x))} \right|_{\mathcal{M}\cap \mathcal{N}}]^{\left( 4,4\right) }$ are $2$-hyponormal.

\end{itemize}

\medskip
The statements (i)--(iv) are established by direct calculations using (\ref{decomposition}) and Lemma \ref{khypo}. \qed
\end{example}

We recall the following open question in \cite[Problem 1.1]{CuYo3}: \ Let $T$
be an operator and let $k\geq 2$. \ Does the $k$-hyponormality of $T$ imply
the $k$-hyponormality of $T^{2}$?

\begin{remark}
\label{Re1} \ (i) \ Consider the classical embedding $W_{\theta(\omega)}$, as in Example \ref{ex103}. \ By Lemma \ref{Lem 2}, we know that $%
W_{\omega }$ is $k$-hyponormal if and only if $W_{\theta(\omega)}$ is (jointly) $k$-hyponormal. \ Moreover, we might expect that Example \ref{ex103}(vi) will provide some clues on the solution of the above mentioned open problem. \newline

(ii) \ For $p,q \ge 1$, let $\mathcal{M}_q:=\vee \{e_{(k_1,k_2)}:k_1 \ge 0 \textrm{ and } k_2 \ge q \}$ and 
$\mathcal{N}_p:=\vee \{e_{(k_1,k_2)}:k_1 \ge p \textrm{ and } k_2 \ge 0 \}$. \ The reader may have noticed that for the classical embedding $W_{\theta(\omega)}$, the restriction 
$\left.W_{\theta(\omega(x))} \right|_{\mathcal{M}_q \cap \mathcal{N}_p}$ is unitarily equivalent to the restriction 
$\left.W_{\theta(\omega(x))} \right|_{\mathcal{M}_{p+q}}$ and also to the restriction 
$\left.W_{\theta(\omega(x))} \right|_{\mathcal{N}_{p+q}}$. \ This strengthens the role of the classical embedding as 
true one-parametric representation of $W_{\omega}$ inside a $2$-variable weighted shift. \newline

(iii) \ For a general operator $T$ on Hilbert space, it is well known that
the hyponormality of $T$ does not imply the hyponormality of $T^{2}$ \cite[Problem 209]{Hal}. \ However, for a unilateral weighted shift $W_{\omega }$, the
hyponormality of $W_{\omega }$ (detected by the condition $\omega _{k}\leq
\omega _{k+1}$ for all $k\geq 0$) clearly implies the hyponormality of every
power $W_{\omega }^{m}\;(m\geq 1)$. $\ $For $2$-variable weighted shifts,
one is thus tempted to expect that a similar result would hold, especially
if we restrict attention to the class of classical embeddings. \
Somewhat surprisingly, it is actually possible to build a hyponormal $2$-variable
weighted shift $W_{\theta(\omega)}$ such
that $W_{\theta(\omega)}^{\left( 2,3\right)}$ is not hyponormal, as shown in Example \ref{ex104}(i).\newline

(iv) \ It is not hard to prove that the classical embedding $\Theta$ of $W_{\omega}^m$ (thought of as the orthogonal direct sum of the classical embeddings of each summand in (\ref{orthosum})) is unitarily equivalent to the power $(m,1)$ or $(1,m)$ of the classical embedding of 
$W_{\omega}$; that is, 
$$
\Theta(W_{\omega}^m):=\bigoplus_{i=0}^{m-1} \; \Theta(W_{\omega \left( m:i \right) }) \cong W_{\theta(\omega)}^{(m,1)}\cong W_{\theta(\omega)}^{(1,m)}.
$$ 

Thus, when it comes to detecting the behavior of positive integer powers of a unilateral weighted shift, to a great extent we can rely on the information provided by the powers $(m,1)$ or $(1,m)$ of its classical embedding. \newline

(v) \ Consider the $2$-variable weighted shift\textit{\ }$%
W_{\left( \alpha ,\beta \right) }\equiv W_{\theta(\omega)}$,
as in Example \ref{ex104}. \ By (iv) above, we observe that $\left.W_{\left( \alpha ,\beta \right)
}\right|_{\mathcal{M}\cap \mathcal{N}}$ and $\left.\left( W_{\left( \alpha ,\beta
\right) }\right|_{\mathcal{M}\cap \mathcal{N}}\right) ^{\left( 3,3\right) }$ are
both also classical embeddings. \ By Lemma \ref{Lem 2}, we know that $\left.W_{\omega }\right|_{\mathcal{L}_{1}}$ is $k$-hyponormal if and only if $\left.W_{\theta(\omega)}\right|_{\mathcal{M}_{1}}$ is (jointly) $k$-hyponormal (all $k\geq 1$), where $\mathcal{L}_{1}$ denotes the subspace of $\ell^2(\mathbb{Z}_+)$ generated by the orthonormal basis vectors $e_1,e_2,\cdots$. \ Example \ref{ex103}(v) provides a clear signal that it might be possible to find an operator $T$ such that the $2$-hyponormality of $T^{2}$ and $T^{3}$ does not imply the $2$-hyponormality of $T$.
\end{remark}



\section{Application 2: The Case of the Bergman Shift} \label{Bergman}

For an answer to Problem \ref{quest}, we recall K. Zhu's lemma. \ For $n\geq 2$%
, let $\mathbb{S}_{n}$ (resp. $\mathbb{B}_{n}$) be unit sphere (resp. ball)
in $\mathbb{C}^{n-1}$. \ We let $dv$ denote the normalized volume measure on $\mathbb{B}%
_{n}$, so that $v\left( \mathbb{B}_{n}\right) =1$. \ The normalized surface
measure on $\mathbb{S}_{n}$ will be denoted by $d\sigma $. 

\begin{lemma}
\label{Lem 1} \cite{Zhu} For $n\geq 2$, let $\mathbb{S}_{n}$ (resp. $\mathbb{%
B}_{n}$) be unit sphere (resp. ball) in $\mathbb{C}^{n-1}$, $\mathbf{m}%
=\left( m_{1},m_{2},\cdots ,m_{n}\right) \in \mathbb{Z}_{+}^{n}$, and $%
\alpha >-1$. \ Let $\mathbf{m!=}m_{1}!m_{2}!\cdots m_{n}!$, $\left\vert
\mathbf{m}\right\vert =m_{1}+m_{2}+\cdots +m_{n}$, and $\Gamma $ be the
gamma function. \ Then
\begin{equation}
\int_{\mathbb{S}_{n}}\left\vert \rho ^{\mathbf{m}}\right\vert ^{2}d\sigma
\left( \rho \right)=\frac{\left( n-1\right) !\mathbf{m!}}{\left(
n-1+\left\vert \mathbf{m}\right\vert \right) !}
\end{equation} \label{Zhu}
and
\begin{equation*}
\int_{\mathbb{B}_{n}}\left\vert z^{\mathbf{m}}\right\vert ^{2}dv_{\alpha
}\left( z\right) =\frac{\mathbf{m!}\Gamma \left( n+\alpha +1\right) }{\Gamma
\left( n+\left\vert \mathbf{m}\right\vert +\alpha +1\right) },
\end{equation*}
where $dv_{\alpha }\left( z\right) =c_{\alpha }\left( 1-\left\vert
z\right\vert ^{2}\right) ^{\alpha }dv\left( z\right) $ and $c_{\alpha }$ is
the a normalizing constant so that $v_{\alpha }\left( \mathbb{B}_{n}\right)
=1 $ and $c_{\alpha }=\frac{\Gamma \left( n+\alpha +1\right) }{n!\Gamma
\left( \alpha +1\right) }$.
\end{lemma}

\begin{theorem} \label{thm62}
(cf. Problem \ref{quest}) \ The Berger measure of $\mathcal{SIE}(B_+)$ is normalized arclength on the line segment $L:=\{(r,1-r): 0 \le r \le 1 \}$, and $\mathcal{SIE}(B_+)$ is unitarily equivalent to the commuting pair $(M_{z_1},M_{z_2})$ of multiplication operators on the Hardy space of the unit sphere $\mathbb{S}_3$. 
\end{theorem}

\begin{proof}
From Figure \ref{Fig 4} we easily see that
$$
\gamma _{(k_{1},k_{2})}(\mathcal{SIE}(B_+))=\frac{%
k_{1}!k_{2}!}{\left( k_{1}+k_{2}+1\right) !}.
$$
On the other hand, these moments agree with those in (\ref{Zhu}), when $n=2$. \ It follows that the weight diagram in Figure \ref{Fig 4} corresponds to the commuting pair $(M_{z_1},M_{z_2})$ of multiplication operators on the Hardy space of the unit sphere $\mathbb{S}_3$. \ Now recall from Theorem \ref{arclengthA} that the Berger measure associated with the weight diagram in Figure \ref{Fig 4} is normalized arclength in the line segment $L$. \ As a consequence, we observe that 
\begin{eqnarray*}
\int_{\mathbb{S}_{n}}\left\vert \rho ^{\mathbf{m}}\right\vert ^{2}d\sigma
\left( \rho \right)=\frac{\mathbf{m!}}{(\left|\mathbf{m}\right|+1)!}=\int_{[0,1]}r^{m_1}(1-r)^{m_2}dr.
\end{eqnarray*}
\end{proof}

{\bf An Alternative Calculation of the Berger Measure of the Agler Shifts}. \ We conclude this section with finding an alternative approach to the calculation of the Berger measure of the Agler $j$-th
shift $A_{j}$ ($j\geq 2$). \ For this, we recall that each component $T_{i}$
of a subnormal $2$-variable weighted shift $W_{\left( \alpha ,\beta \right)
}\equiv (T_{1},T_{2})$ is subnormal. \ For instance, $T_{1}\cong
\bigoplus_{k_{2}=0}^{\infty }W_{k_{2}}$ (resp. $T_{2}\cong
\bigoplus_{k_{1}=0}^{\infty }V_{k_{1}}$). \ In this case, we let $\{\xi
_{k_{2}}\}_{k_{2}=0}^{\infty }$ and $\{\eta _{k_{1}}\}_{k_{1}=0}^{\infty }$
be the $1$-variable Berger measures for $%
\{W_{k_{2}}\}_{k_{2}=0}^{\infty }$ and $\{V_{k_{1}}\}_{k_{1}=0}^{\infty }$,
respectively. \ Fix $k_{2}\geq 0$ and observe that the moments of $\xi
_{k_{2}}$ are given by
\begin{eqnarray}
\int_{[0,\left\|T_1\right\|]}s^{k_{1}}\;d\xi_{k_{2}}(s)&=&\alpha _{0,k_{2}}^{2}\cdot \ldots \cdot
\alpha _{k_{1}-1,k_{2}}^{2}=\frac{\gamma _{(k_{1},k_{2})}}{\gamma _{(0,k_{2})}} \nonumber \\
&=&\frac{1}{\gamma _{(0,k_{2})}}\int_{[0,\left\|T_1\right\|] \times [0,\left\|T_2\right\|]}s^{k_{1}}t^{k_{2}}\;d\mu (s,t)\;\;(\text{%
all }k_{1}\geq 0).  \label{oldeq}
\end{eqnarray}
We also know that if $\mu$ is the Berger measure of $W_{(\alpha,\beta)}$, then $\xi_{k_{2}}=\mu _{k_{2}}^{X}$%
, where $d\mu _{k_{2}}(s,t)=\frac{1}{\gamma _{0k_{2}}}t^{k_{2}}d\mu (s,t)$ (%
\cite[Theorem 3.1]{CuYo2}). \ 

Recall now the Drury-Arveson $2$-variable weighted
shift $DA$, whose weight sequences are given by

\begin{eqnarray}
\alpha _{(k_{1},k_{2})}:= &&\sqrt{\frac{k_{1}+1}{k_{1}+k_{2}+1}}\;\;(\text{%
for all}\;k_{1},k_{2}\geq 0)\medskip  \label{alpha} \\
\beta _{(k_{1},k_{2})}:= &&\alpha _{(k_{2},k_{1})}\;\;(\text{for all}%
\;k_{1},k_{2}\geq 0).  \label{beta}
\end{eqnarray}

If we denote the successive rows of the weight diagram of $DA$ by $%
R_{0}, R_{1},R_{2},\cdots $, it is easy to see that $R_{0}=A_{1}$, the (unweighted)
unilateral shift, $R_{1}=A_{2}\equiv B_+$, the Bergman shift and, more generally, $%
R_{k_{2}-1}=A_{k_{2}}$, the Agler $k_{2}$-th shift ($k_{2}\geq 2$); in
particular, all rows and columns are subnormal weighted shifts. \ For $%
k_{2}\geq 2$, it is well known that the Berger measure of $A_{k_{2}}$ is $%
(k_{2}-1)(1-r)^{k_{2}-2}dr$ on the closed interval $[0,1]$ (cf. \cite[paragraph immediately preceding Theorem 1.1]%
{CuEx}). \ However, $DA$ is {\it not} subnormal (not even hyponormal!), so we may not directly use the above argument. \ On the other hand, $DA$ is a close relative of the spherically isometric embedding $\mathcal{SIE}(A_2)$ of the Bergman shift. \ That is, each row $R_j \; (j \ge 1)$ in $DA$ appears as row $R_{j-1}$ in $\mathcal{SIE}(A_2)$. \ We thus obtain an alternative way to calculate the Berger measure of the Agler shifts.

\begin{corollary} \label{cor63}
For $k_{2}\geq 2$, the Berger measure of $A_{k_{2}}$ is $%
(k_{2}-1)(1-r)^{k_{2}-2}dr $ on $[0,1]$.
\end{corollary}



\setlength{\unitlength}{1mm} \psset{unit=1mm}
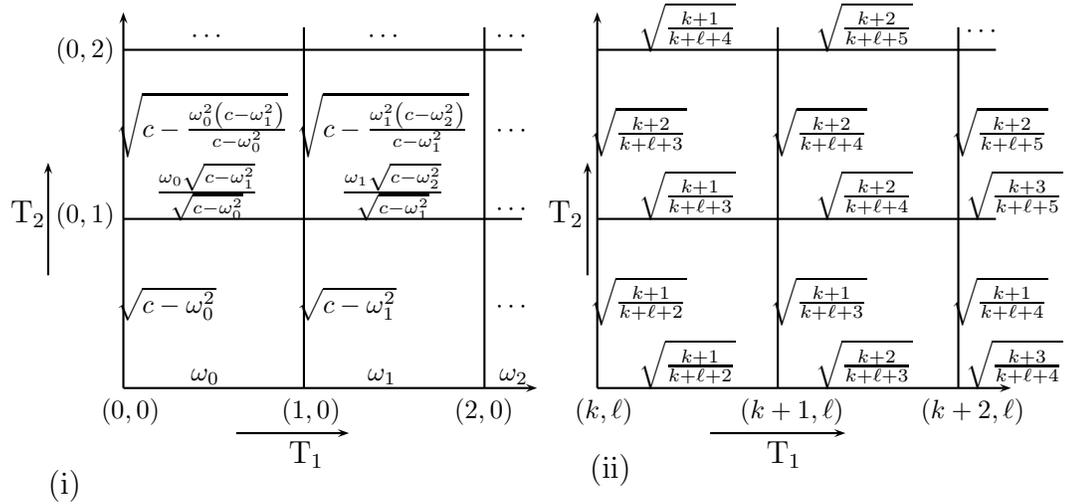
\begin{figure}[th]
\begin{center}
\begin{picture}(135,70)

\put(6,42){\footnotesize{$(0,1)$}}
\put(6,64){\footnotesize{$(0,2)$}}

\psline{->}(30,14)(45,14)
\put(37,10){$\rm{T}_1$}
\psline{->}(5,35)(5,50)
\put(0,42){$\rm{T}_2$}

\psline{->}(15,20)(70,20)
\psline(15,42.5)(68,42.5)
\psline(15,65)(68,65)

\psline{->}(15,20)(15,70)
\psline(39,20)(39,68)
\psline(63,20)(63,68)

\put(12,16){\footnotesize{$(0,0)$}}
\put(36,16){\footnotesize{$(1,0)$}}
\put(59,16){\footnotesize{$(2,0)$}}

\put(24,21){\footnotesize{$\omega_{0}$}}
\put(47.5,21){\footnotesize{$\omega_{1}$}}
\put(65,21){\footnotesize{$\omega_{2}$}}

\put(19.3,45.2){\footnotesize{$\frac{\omega _{0}\sqrt{c-\omega_{1}^2}}{\sqrt{c-\omega_{0}^2}}$}}
\put(44,45.2){\footnotesize{$\frac{\omega _{1}\sqrt{c-\omega_{2}^2}}{\sqrt{c-\omega_{1}^2}}$}}
\put(64.5,43){\footnotesize{$\cdots$}}

\put(14,30){\footnotesize{$\sqrt{c-\omega _{0}^2}$}}
\put(38,30){\footnotesize{$\sqrt{c-\omega _{1}^2}$}}
\put(64.5,30){\footnotesize{$\cdots$}}

\put(14,53.5){\footnotesize{$\sqrt{c-\frac{\omega _{0}^{2}\left( c-\omega_{1}^2\right) }{c-\omega_{0}^2}}$}}
\put(38,53.5){\footnotesize{$\sqrt{c-\frac{\omega _{1}^{2}\left( c-\omega_{2}^2\right) }{c-\omega_{1}^2}}$}}
\put(64.5,53.5){\footnotesize{$\cdots$}}

\put(24,66){\footnotesize{$\cdots$}}
\put(47.5,66){\footnotesize{$\cdots$}}
\put(64.5,66){\footnotesize{$\cdots$}}

\put(5,6){(i)}


\put(77,8){(ii)}

\psline{->}(93,14)(111,14)
\put(100.5,10){$\rm{T}_{1}$}
\psline{->}(76.7,35)(76.7,50)
\put(71.5,42){$\rm{T}_{2}$}

\psline{->}(78,20)(133,20)
\psline(78,42.5)(131,42.5)
\psline(78,65)(131,65)

\psline{->}(78,20)(78,70)
\psline(102,20)(102,68)
\psline(126,20)(126,68)

\put(75,16){\footnotesize{$(k,\ell)$}}
\put(97,16){\footnotesize{$(k+1,\ell)$}}
\put(121,16){\footnotesize{$(k+2,\ell)$}}

\put(83.5,22.1){\footnotesize{$\sqrt{\frac{k+1}{k+\ell+2}}$}}
\put(107,22.1){\footnotesize{$\sqrt{\frac{k+2}{k+\ell+3}}$}}
\put(127,22.1){\footnotesize{$\sqrt{\frac{k+3}{k+\ell+4}}$}}

\put(83.5,44.6){\footnotesize{$\sqrt{\frac{k+1}{k+\ell+3}}$}}
\put(107,44.6){\footnotesize{$\sqrt{\frac{k+2}{k+\ell+4}}$}}
\put(127,44.6){\footnotesize{$\sqrt{\frac{k+3}{k+\ell+5}}$}}

\put(77,30.5){\footnotesize{$\sqrt{\frac{k+1}{k+\ell+2}}$}}
\put(101,30.5){\footnotesize{$\sqrt{\frac{k+1}{k+\ell+3}}$}}
\put(125,30.5){\footnotesize{$\sqrt{\frac{k+1}{k+\ell+4}}$}}

\put(77,53){\footnotesize{$\sqrt{\frac{k+2}{k+\ell+3}}$}}
\put(101,53){\footnotesize{$\sqrt{\frac{k+2}{k+\ell+4}}$}}
\put(125,53){\footnotesize{$\sqrt{\frac{k+2}{k+\ell+5}}$}}

\put(83.5,67){\footnotesize{$\sqrt{\frac{k+1}{k+\ell+4}}$}}
\put(107,67){\footnotesize{$\sqrt{\frac{k+2}{k+\ell+5}}$}}
\put(127,66.7){\footnotesize{$\cdots$}}

\end{picture}
\end{center}
\caption{(i) \ Weight diagram of a spherically quasinormal embedding of $W_{ \omega }$; (ii) \ weight diagram 
of $\left.W_{\theta(\omega)}\right|_{\mathcal{M}_k \cap \mathcal{N}_{\ell}}$ when $W_{\omega }$ is the Bergman shift.}
\label{Fig 2}
\end{figure}


\begin{proof}
Consider Figure \ref{Fig 2}(ii) and Figure \ref{Fig 4} with $W_{\omega }=B_{+}$. \
Since $W_{\left( \alpha(\omega),\beta(\omega)\right) }$ is
spherically quasinormal, it is also subnormal. \ Let $\mu $ be the Berger measure of 
$W_{\left( \alpha(\omega),\beta(\omega)\right) }$ and let $\sigma_{k_{2}}$ be the
Berger measure of $A_{k_{2}}$, respectively. \ Recall that
\begin{equation*}
\begin{tabular}{l}
$\sigma_{k_{2}}=\mu _{k_{2}-2}^{X}$, \; where \; $d\mu _{k_{2}}(s,t)=\frac{1}{\gamma
_{0k_{2}}}t^{k_{2}}d\mu (s,t)$.
\end{tabular}%
\end{equation*}%
Also, for $k_2 \ge 2$, $\sigma_{k_2}=\xi_{k_2-2}$. \ By Theorem \ref{Th4}, Figure \ref{Fig 2}(ii) and Figure \ref{Fig 4},
for all $k_{2}\geq 0$ and $\left( s,t\right) \in
G\left( f\right) =\left\{ \left( r,1-r\right) :r\in \left[ 0,1\right]
\right\} $, we can see that%
\begin{equation*}
d\mu _{k_{2}}(s,t)=\frac{1}{\gamma _{0k_{2}}}t^{k_{2}}d\mu (s,t)=\left(
k_{2}+1\right) t^{k_{2}}d\mu (s,t)=\left( k_{2}+1\right) (1-r)^{k_{2}}d\mu\left(r, 1-r\right).
\end{equation*}%
Thus, we have
\begin{equation} \label{eqk2}
d\xi _{k_{2}}\left( r\right) =d \left( \mu _{k_{2}}\left( s,t\right) \right)^X =\left(
k_{2}+1\right) \left( \left( 1-r\right) ^{k_{2}}d\mu \left(r,1-r\right)\right) ^{X}=(k_2+1)\left( 1-r\right) ^{k_{2}}d\sigma(r),
\end{equation}
where we have used Corollary \ref{newcor} to calculate the marginal measure. \ Therefore, the Berger measure of $A_{k_2}$ is $(k_{2}-1)(1-r)^{k_{2}-2}dr$
on the closed interval $[0,1]$. \ This completes the proof.
\end{proof}

\begin{remark}
(i) \ In the Proof of Corollary \ref{cor63} we chose to calculate the Berger measure of the Agler shifts by using the rows of the spherically isometric embedding. \ However, since the weight diagram in Figure \ref{Fig 4} is symmetric with respect to the diagonal, we could have chosen columns instead of rows, to perform the calculation. \ In that case, one gets
\begin{equation*}
d\eta _{k_{1}}\left( r\right) =d \left( \mu _{k_{1}}\left( s,t\right) \right)^Y =\left(
k_{1}+1\right) \left( r^{k_{1}}d\mu \left(r,1-r\right)\right) ^{Y}=(k_1+1)r^{k_{1}}d\sigma(1-r).
\end{equation*}
The substitution $u:=1-r$ shows that this answer is identical to that in (\ref{eqk2}). \newline
(ii) \ The reader may have noticed that our calculation of the Berger measure of the Agler shifts in Corollary \ref{cor63} is based on the theory of spherically isometric embeddings and the disintegration-of-measure techniques from \cite{CuYo2}. \ While this approach bypasses K. Zhu's result (Lemma \ref{Zhu}), it does not provide a formula for the measure $\mu$ (as we did in Theorem \ref{thm62}).
\end{remark}


\section{Application 3: Monomial Pairs} \label{examples}

In this section we first present an alternative solution to Problem \ref{prob1}(ii), using the theory of recursively generated weighted shifts. \ We then discuss two examples that illustrate the solution to Problem \ref{prob1}.

\begin{theorem}
(cf. Problem \ref{prob1}(ii)) \ Let $W_0$, $\sigma$ and $\mu$ be given as in Problem \ref{prob1}, assume that $W_0$ is recursively generated, with 
$$
\sigma=\rho_0 \delta_{s_0}+\cdots+\rho_{r-1}\delta_{s_{r-1}}
$$
with $0 \le s_{0}< \cdots < s_{r-1}\le c$, $\rho_0,\cdots \rho_{r-1}>0$, and $\rho_0+\cdots + \rho_{r-1}=1$. \ Then 
$$
\mu=\rho_0 \delta_{(s_0,c-s_0)}+\cdots+\rho_{r-1}\delta_{(s_{r-1},c-s_{r-1})}.
$$
\end{theorem}

\begin{proof}
Since $W_{0}:=\shift (\alpha _{\left( 0,0\right) },\alpha _{\left(
1,0\right) },\cdots )$ is a recursively generated weighted shift, by Lemma %
\ref{Lem 0} there exists a generating function $g_{0}\left( s\right) $ such
that the zeros of $g_{0}$ are the atoms of the Berger
measure of $W_{0}$. \ That is, if we let
\begin{equation}
g_{0}\left( s\right) =s^{r}-\left( \varphi _{r-1}s^{r-1}+\varphi
_{r-2}s^{r-2}+\cdots +\varphi _{1}s+\varphi _{0}\right) \text{,}  \label{1}
\end{equation}%
then $g_{0}\left( s\right) $ has $r$ distinct real roots $%
0<s_{0}<s_{1}<\cdots <s_{r-1} \le c$. \ Since $\alpha _{(k_{1},k_{2})}^{2}+\beta
_{(k_{1},k_{2})}^{2}=c>0$ for every $(k_{1},k_{2})\in \mathbb{Z}_{+}^{2}$,
we have
\begin{equation*}
\gamma _{(k_{1}+1,k_{2})}(W_{(\alpha ,\beta )})+\gamma
_{(k_{1},k_{2}+1)}(W_{(\alpha ,\beta )})=c\cdot \gamma
_{(k_{1},k_{2})}(W_{(\alpha ,\beta )})\text{.}
\end{equation*}%
It follows that we have a generating function $h\left( s,t\right) $ of $%
W_{(\alpha ,\beta )}$, given by
\begin{equation}
h\left( s,t\right):=s+t-c\text{.}  \label{2}
\end{equation}%
Thus, the atoms of the Berger measure $\mu $ for $W_{(\alpha ,\beta )}$ are
also zeros of $h$. \ Let $\tau $ be the Berger measure of $%
V_{0}:=\shift (\beta _{\left( 0,0\right) },\beta _{\left( 0,1\right)},\beta _{\left( 0,2\right)},\cdots )$. \ By \cite[Corollary 5.6]{CuYo7}
\begin{equation*}
\supp \; \tau =c-\supp \; \sigma :=\{c-s:s\in \supp \; \sigma \}%
\text{.}
\end{equation*}
It follows that the atoms of $\mu $ must be of the form $\left\{ \left(
s_{i},c-s_{i}\right) :0\leq i\leq r-1\right\} $. \ We now denote the densities of $\mu$ by $\xi_{i}$, so that
\begin{equation}
\mu =\xi _{0}\delta _{\left( s_{0},c-s_{0}\right) }+\cdots +\xi
_{r-1}\delta _{\left( s_{r-1},c-s_{r-1}\right) }\text{,}  \label{3}
\end{equation}%
where $\xi _{0},\xi _{1},\cdots ,\xi _{r-1}$ are positive real numbers which add up to $1$. \ We thus have
\begin{equation}
\begin{tabular}{l}
$\xi _{0}+\xi _{1}+\cdots +\xi _{r-1}=\gamma _{\left( 0,0\right) }=1$%
;\medskip \\
$\xi _{0}s_{0}+\xi _{1}s_{1}+\cdots +\xi _{r-1}s_{r-1}=\gamma _{\left(
1,0\right) }$;\medskip \\
$\xi _{0}s_{0}^2 +\xi _{1}s_{1}^2 +\cdots
+\xi _{r-1}s_{r-1}^2 =\gamma _{\left( 2,0\right) }$; \\
$\ \ \ \ \ \ \ \ \ \ \ \ \ \ \ \ \ \ \vdots $ \\
$\xi _{0}s_{0}^{r-1}+\xi _{1}s_{1}^{r-1}+\cdots +\xi _{r-1}s_{r-1}^{r-1}=\gamma _{\left(r-1,0\right) }$,
\end{tabular}
\label{4}
\end{equation}
Observe that from (\ref{4}) we obtain an
invertible $r\times r$ matrix
\begin{equation*}
V:=\left(
\begin{array}{cccc}
1 & 1 & \cdots & 1 \\
s_{0} & s_{1} & \cdots & s_{r-1} \\
s_{0}^{2} & s_{1}^{2} & \cdots & s_{r-1}^{2} \\
s_{0}^3 & s_{1}^3 & \cdots &
s_{r-1}^3 \\
\vdots & \vdots & \ddots & \vdots \\
s_{0}^{r-1} & s_{1}^{r-1} &
\cdots & s_{r-1}^{r-1}
\end{array}%
\right) _{\left( r\times r\right) }\text{.}
\end{equation*}
Then, we have
$$
\left(
\begin{array}{c}
\xi _{0} \\
\vdots \\
\xi _{r-1}%
\end{array}%
\right) =V^{-1}\left(
\begin{array}{c}
\gamma _{\left( 0,0\right) } \\
\gamma _{\left( 1,0\right) } \\
\gamma _{\left( 2,0\right) } \\
\gamma _{\left( 3,0\right) } \\
\vdots \\
\gamma _{(r-1,0)}%
\end{array}%
\right) _{\left( 1\times r\right) }
=V^{-1}
\left(
\begin{array}{c}
\gamma _{0}[\sigma] \\
\gamma _{1}[\sigma] \\
\gamma _{2}[\sigma] \\
\gamma _{3}[\sigma] \\
\vdots \\
\gamma _{r-1}[\sigma]
\end{array}
\right)
=
\left(
\begin{array}{c}
\rho _{0} \\
\vdots \\
\rho _{r-1}%
\end{array}%
\right).
$$
Thus, we have
\begin{equation} \label{mueq}
\mu=\rho_0 \delta_{(s_0,c-s_0)}+\cdots+\rho_{r-1}\delta_{(s_{r-1},c-s_{r-1})}.
\end{equation}
Therefore, (\ref{mueq}) provides an answer to Problem \ref{prob1} (both (i) and (ii)).
\end{proof}

\begin{example}
\label{Ex1}Consider a contractive spherically quasinormal $2$-variable
weighted shift $W_{(\alpha,\beta)}\equiv (T_{1},T_{2})$ given by
Figure \ref{Fig 2}(i), where $W_{\omega}:=\shift (\omega_0,\omega_1,\cdots )$ is subnormal with Berger
measure $\sigma:=\frac{1}{3}\delta _{\frac{1}{3}}+\frac{1}{3}\delta _{%
\frac{1}{2}}+\frac{1}{3}\delta _{1}$. \ Then $W_{(\alpha,\beta)}$ is subnormal with
$3$-atomic Berger measure $\mu $ which is given by%
\begin{equation*}
\mu =\frac{1}{3}\delta _{\left( \frac{1}{3},\frac{2}{3}\right) }+\frac{1}{3}%
\delta _{\left( \frac{1}{2},\frac{1}{2}\right) }+\frac{1}{3}\delta _{\left(
1,0\right) }\text{,}
\end{equation*}%
where $c=1$.

To see this, note that
\begin{equation*}
V=\left(
\begin{array}{ccc}
1 & 1 & 1 \\
s_{0} & s_{1} & s_{2} \\
s_{0}^{2} & s_{1}^{2} & s_{2}^{2}%
\end{array}%
\right) =\left(
\begin{array}{ccc}
1 & 1 & 1 \\
\frac{1}{3} & \frac{1}{2} & 1 \\
\frac{1}{9} & \frac{1}{4} & 1%
\end{array}%
\right) \text{ and }V^{-1}=\left(
\begin{array}{ccc}
\frac{9}{2} & -\frac{27}{2} & 9 \\
-4 & 16 & -12 \\
\frac{1}{2} & -\frac{5}{2} & 3%
\end{array}%
\right)
\end{equation*}%
Thus, we have%
\begin{equation*}
\left(
\begin{array}{c}
\rho _{0} \\
\rho _{1} \\
\rho _{2}%
\end{array}%
\right) =\left(
\begin{array}{ccc}
\frac{9}{2} & -\frac{27}{2} & 9 \\
-4 & 16 & -12 \\
\frac{1}{2} & -\frac{5}{2} & 3%
\end{array}%
\right) \left(
\begin{array}{c}
1 \\
\frac{11}{18} \\
\frac{49}{36}%
\end{array}%
\right) =\left(
\begin{array}{c}
\frac{1}{3} \\
\frac{1}{3} \\
\frac{1}{3}%
\end{array}%
\right) \text{,} 
\end{equation*}
as desired. \qed
\end{example}

Our last example shows how the Berger measure of a spherically quasinormal embedding can sometimes be calculated by elementary methods, bypassing Theorem \ref{Prop1}.

\begin{example}
\label{Quasinormal5}Consider the spherically quasinormal $2$-variable
weighted shift $W_{(\alpha,\beta)} \equiv (T_{1},T_{2})$ given by
Figure \ref{Fig 2}(i), where $W_{\omega}:=\shift (\omega_0, \omega_1, \cdots )$ is subnormal with Berger
measure $\sigma:=x\delta _{a}+y\delta _{b}+\left( 1-x-y\right) \delta _{1}$
($0\leq a<b<1$, $0<x,y<1$, and $x+y<1$). \ In Figure \ref{Fig 2}(i), assume that $c \ge 1$. \ Then $W_{(\alpha,\beta)}$ is subnormal with a three-atomic
Berger measure $\mu $ given by%
\begin{equation*}
\mu =x\delta _{\left( a,c-a\right) }+y\delta _{\left( b,c-b\right) }+\left(
1-x-y\right) \delta _{\left( 1,c-1\right) }\text{.}
\end{equation*}

Indeed, it follows from (\ref{2}) that $W_{(\alpha,\beta)}$ is subnormal with three-atomic
Berger measure. \ Thus, we have%
\begin{equation}
\mu =\rho_0\delta _{\left( a,c-a\right) }+\rho_1\delta _{\left( b,c-b\right) }+\left(
1-\rho_0-\rho_1 \right) \delta _{\left( 1,c-1\right) }\text{,}  \label{17}
\end{equation}%
and recall that $0<x,y<1$ and $x+y<1$. \ Since $0\leq a<b<1$, (\ref{Berger Theorem})
and (\ref{17}) imply that%
\begin{equation}
\lim_{k_{1}\rightarrow \infty }\gamma _{(k_{1},0)}\left(W_{(\alpha,\beta)}\right)
=1-\rho_0-\rho_1=\lim_{k_{1}\rightarrow \infty }\gamma _{k_{1}}\left( W_{\omega}\right)
=1-x-y \text{.}  
\end{equation}%
It follows readily that 
\begin{equation}
\rho_0+\rho_1=x+y. \label{21}
\end{equation}
By (\ref{17}) and (\ref{21}), for all $k_{1}\geq 0$, we have%
\begin{equation*}
\gamma _{(k_{1},0)}\left( W_{(\alpha,\beta)}\right) =\gamma _{k_{1}}\left(
W_{\omega}\right) \Longrightarrow \left( \rho_0-x\right) a^{k_{1}}+\left( \rho_1-y\right)
b^{k_{1}}=0\text{.}
\end{equation*}%
If $\rho_0-x\neq 0$, then
\begin{equation*}
0=\lim_{k_{1}\rightarrow \infty }\left( \frac{a}{b}\right) ^{k_{1}}=\frac{y-n}{%
m-x}\Longrightarrow n=y\text{,}
\end{equation*}%
which is a contradiction to (\ref{21}). \ Thus, we must have $\rho_0=x$, and a fortiori $\rho_1=y$. \ It follows that 
\begin{equation}
\mu =x\delta _{\left( a,c-a\right) }+y\delta _{\left( b,c-b\right) }+\left(
1-x-y\right) \delta _{\left( 1,c-1\right) }\text{,}  \label{measure 2}
\end{equation}%
as desired. \ Note in passing that the marginal measure $\mu ^{X}$ is $\sigma$. \qed
\end{example}


\section{Appendix} \label{appendix}

\subsection{Unilateral weighted shifts} \label{sub25}
\ For $\omega \equiv \{\omega _{n}\}_{n=0}^{\infty }$ a bounded sequence of
positive real numbers (called \textit{weights}), let $W_{\omega }\equiv
\shift (\omega _{0},\omega _{1},\cdots ):\ell ^{2}(\mathbb{Z}%
_{+})\rightarrow \ell ^{2}(\mathbb{Z}_{+})$ be the associated unilateral
weighted shift, defined by $W_{\omega }e_{n}:=\omega _{n}e_{n+1}\;$(all $%
n\geq 0$), where $\{e_{n}\}_{n=0}^{\infty }$ is the canonical orthonormal
basis in $\ell ^{2}(\mathbb{Z}_{+})$. \ The {\it moments} of $\omega \equiv
\{\omega _{n}\}_{n=0}^{\infty }$ are given as
\begin{equation}
\gamma _{k}\equiv \gamma _{k}(W_{\omega }):=\left\{
\begin{tabular}{ll}
$1$, & $\text{if }k=0$ \\
$\omega _{0}^{2}\cdots \omega _{k-1}^{2}$, & $\text{if }k \ge 1.$%
\end{tabular}%
\right.  \label{moment10}
\end{equation}
\medskip
The (unweighted) unilateral shift is $U_+:=\shift (1,1,1,\cdots)$, and for $0<a<1$ we let $S_{a}:=\shift (a,1,1,\cdots )$.

\medskip

We now recall a well known characterization of
subnormality for unilateral weighted shifts, due to C. Berger
(cf. \cite[III.8.16]{Con}) and independently established by Gellar and
Wallen (\cite{GeWa}): \ $W_{\omega }$ is subnormal if and only if there exists a probability measure $%
\sigma $ supported in $[0,\left\| W_{\omega }\right\| ^{2}]$ (called the 
\textit{Berger measure} of $W_{\omega }$) such that 
$$
\gamma _{k}(W_{\omega})=\omega _{0}^{2}\cdot ...\cdot \omega _{k-1}^{2}=\int t^{k}d\sigma (t) \; (k\geq 1).
$$

\medskip
\ Observe that $U_{+}$ and $S_{a}$ are subnormal, with Berger measures $\delta _{1}$ and $%
(1-a^{2})\delta _{0}+a^{2}\delta _{1}$, respectively, where $\delta _{p}$
denotes the point-mass probability measure with support the singleton set $%
\{p\}$. \ On the other hand, the Berger measure of the Bergman shift $B_{+}$ (acting on $A^2(\mathbb{D})$, and with weights $\omega_{n}:=\sqrt{\frac{n+1}{n+2}}\;(n\geq 0)$) is the Lebesgue measure on the interval $[0,1]$.


\subsection{Recursively generated unilateral weighted shifts} \label{RGWS}

We close this section with recalling some terminology and basic results from
\cite{CuFi1} and \cite{tcmp1}. \ A subnormal unilateral weighted shift $%
W_{\omega }$ is said to be \textit{recursively generated} if the sequence of
moments $\gamma _{n}$ admits a finite-step recursive relation; that is, if
there exists an integer $k\geq 1$ and real coefficients $\varphi
_{0},\varphi _{1},\cdots ,\varphi _{k-1}$ such that
\begin{equation}
\gamma _{n+k}=\varphi _{0}\gamma _{n}+\varphi _{1}\gamma _{n+1}+\cdots
+\varphi _{k-1}\gamma _{n+k-1}\;\;(\text{all }n\geq 0).  \label{receq}
\end{equation}%
In conjunction with (\ref{receq}) we consider the generating function
\begin{equation}
g_{\omega }(s):=s^{k}-(\varphi _{0}+\varphi _{1}s+\cdots +\varphi
_{k-1}s^{k-1}).  \label{receq2}
\end{equation}

The following result characterizes recursively generated subnormal
unilateral weighted shifts.

\begin{lemma}
\label{Lem 0} \cite{tcmp1} Let $W_{\omega }$ be a subnormal unilateral
weighted shift. \ The following statements are equivalent. \newline
(i) \ $W_{\omega }$ is recursively generated. \medskip \newline
(ii) \ The Berger measure $\sigma $ of $W_{\omega }$ is finitely atomic, and 
$g_{\omega}(s)=\prod_i (s-s_i)$, where $\supp \; \sigma = \{s_0,\cdots,s_{k-1}\}$.
\end{lemma}


\subsection{$2$-variable weighted shifts} \label{sub30}

Consider now two double-indexed positive bounded sequences $\alpha _{%
\mathbf{k}},\beta _{\mathbf{k}}\in \ell ^{\infty }(\mathbb{Z}_{+}^{2})$, $%
\mathbf{k}\equiv (k_{1},k_{2})\in \mathbb{Z}_{+}^{2}$ and let $\ell ^{2}(%
\mathbb{Z}_{+}^{2})$\ be the Hilbert space of square-summable complex
sequences indexed by $\mathbb{Z}_{+}^{2}$. \ (Recall that $\ell ^{2}(\mathbb{%
Z}_{+}^{2})$ is canonically isometrically isomorphic to $\ell ^{2}(\mathbb{Z}%
_{+})\bigotimes \ell ^{2}(\mathbb{Z}_{+})$.) \ We define the $2$-variable
weighted shift $W_{(\alpha,\beta)}\equiv (T_{1},T_{2})$ as a pair of operators given by 
\begin{equation} \label{defT}
T_{1}e_{\mathbf{k}}:=\alpha _{\mathbf{k}}e_{\mathbf{k+}\varepsilon _{1}}%
\text{ and }T_{2}e_{\mathbf{k}}:=\beta _{\mathbf{k}}e_{\mathbf{k+}%
\varepsilon _{2}}\text{,}
\end{equation}
where $\mathbf{\varepsilon }_{1}:=(1,0)$ and $\mathbf{\varepsilon }%
_{2}:=(0,1)$. \ Clearly,
\begin{equation}
T_{1}T_{2}=T_{2}T_{1}\Longleftrightarrow \beta _{\mathbf{k+}\varepsilon
_{1}}\alpha _{\mathbf{k}}=\alpha _{\mathbf{k+}\varepsilon _{2}}\beta _{%
\mathbf{k}}\;\left( \text{all }\mathbf{k}\in \mathbb{Z}_{+}^{2}\right) .
\label{commuting}
\end{equation}%
Moreover, for $\mathbf{k} \in \mathbb{Z}_{+}^{2}$ we have
\begin{eqnarray} \label{adjoint}
T_{1}^*e_{0,k_2}=0 & \textrm{ and } & T_{1}^*e_{\mathbf{k}}=\alpha _{\mathbf{k-\varepsilon_{1}}}e_{\mathbf{k-}\varepsilon _{1}} \; (k_1 \ge 1); \\
T_{2}^*e_{k_1,0}=0 & \textrm{ and } & T_{2}^*e_{\mathbf{k}}:=\beta _{\mathbf{k-\varepsilon_{2}}}e_{\mathbf{k-\varepsilon_{2}}} \; (k_2 \ge 1).
\end{eqnarray}
In an entirely similar way one can define multivariable weighted shifts. \ The weight diagram of a generic $2$-variable weighted shift is shown in Figure \ref{Fig 1}(i). \ 

When all weights are equal to $1$ we obtain the so-called Helton-Howe shift; that is, the shift that corresponds to the pair of multiplications by the coordinate functions in the Hardy space $H^2(\mathbb{T} \times \mathbb{T})$ of the $2$-torus, with respect to normalized arclength measure on each unit circle $\mathbb{T}$ (cf. \cite{Gle}). \ This shift can also be represented as $(U_+ \otimes I, I \otimes U_+)$, acting on $\ell^2(\mathbb{Z}_+) \otimes \ell^2(\mathbb{Z}_+)$.

\medskip
We now recall the definition of {\it moments} for a commuting $2$-variable weighted shift $W_{(\alpha,\beta)}$. \ Given $\mathbf{k}%
\equiv (k_{1},k_{2})\in \mathbb{Z}_{+}^{2}$, the moment of $W_{(\alpha ,\beta )}$ of order $\mathbf{k}$ is
\begin{equation}
\gamma _{\mathbf{k}}\equiv \gamma _{\mathbf{k}}(W_{(\alpha ,\beta )}):=%
\begin{cases}
1, & \text{if }k_{1}=0\text{ and }k_{2}=0 \\
\alpha _{(0,0)}^{2}\cdots \alpha _{(k_{1}-1,0)}^{2}, & \text{if }k_{1}\geq 1%
\text{ and }k_{2}=0 \\
\beta _{(0,0)}^{2}\cdots \beta _{(0,k_{2}-1)}^{2}, & \text{if }k_{1}=0\text{
and }k_{2}\geq 1 \\
\alpha _{(0,0)}^{2}\cdots \alpha _{(k_{1}-1,0)}^{2}\beta
_{(k_{1},0)}^{2}\cdots \beta _{(k_{1},k_{2}-1)}^{2}, & \text{if }k_{1}\geq 1%
\text{ and }k_{2}\geq 1.%
\end{cases}
\label{moment0}
\end{equation}%
We remark that, due to the commutativity condition (\ref{commuting}), $%
\gamma _{\mathbf{k}}$ can be computed using any nondecreasing path from $%
(0,0)$ to $(k_{1},k_{2})$. 

To detect hyponormality, there is a simple criterion:

\begin{theorem}
(\cite{bridge}) \ (Six-point Test) \ Let $W_{(\alpha,\beta)}$ be a $2$%
-variable weighted shift, with weight sequences $\alpha $ and $\beta $. \
Then 
\begin{equation*}
W_{(\alpha,\beta)} \text{ is hyponormal }\Leftrightarrow \left( 
\begin{array}{cc}
\alpha _{\mathbf{k}+\mathbf{\varepsilon }_{1}}^{2}-\alpha _{\mathbf{k}}^{2}
& \alpha _{\mathbf{k}+\mathbf{\varepsilon }_{2}}\beta _{\mathbf{k}+\mathbf{%
\varepsilon }_{1}}-\alpha _{\mathbf{k}}\beta _{\mathbf{k}} \\ 
\alpha _{\mathbf{k}+\mathbf{\varepsilon }_{2}}\beta _{\mathbf{k}+\mathbf{%
\varepsilon }_{1}}-\alpha _{\mathbf{k}}\beta _{\mathbf{k}} & \beta _{\mathbf{%
k}+\mathbf{\varepsilon }_{2}}^{2}-\beta _{\mathbf{k}}^{2}%
\end{array}
\right) \geq 0\; \\ \;(\text{all }\mathbf{k}\in \mathbb{Z}_{+}^{2}).
\end{equation*}
\end{theorem}

A straightforward generalization of the above mentioned Berger-Gellar-Wallen result was proved in \cite{JeLu}. \ That is, a commuting pair $W_{(\alpha,\beta)}$ 
admits a commuting normal extension if and only if there is a
probability measure $\mu $ (which we call the Berger measure of $W_{(\alpha,\beta)}$)
 defined on the $2$-dimensional rectangle $R=[0,a_{1}]\times \lbrack
0,a_{2}]$ (where $a_{i}:=\left\| T_{i}\right\| ^{2}$) such that 
\begin{equation} \label{Berger Theorem}
\gamma _{\mathbf{k}}(W_{(\alpha,\beta)})
=\int t_{1}^{k_{1}}t_{2}^{k_{2}}\;d\mu (t_{1},t_{2})\;\;(\text{all }
\mathbf{k}\in \mathbb{Z}_+^2).
\end{equation}
\medskip

Thus, the study of subnormality for multivariable weighted shifts is
intimately connected to multivariable real moment problems. \ (For more on $2$-variable weighted shifts, see \cite{Cur}.)


\subsection{Disintegration of measures} \label{disintegration}

The theory of disintegration of measures is well known, and has been used in operator theory for more than forty years, especially in the study of composition operators, where the expectation operator has become a fundamental research tool. \ Most of the discussion in this Subsection is taken from \cite[VII.2, pp. 317-319]{Con}, \cite{CuYo2} and \cite{Yoo}. \ Let
$X$ and $Z$ be compact metric spaces and let $\mu $ be a positive regular
Borel measure on $Z$ and the space of all Borel measures on $Z$ is denoted
by $M(Z)$. \ For $\phi :Z\rightarrow X$ a Borel mapping, let $\nu $ be the
Borel measure $\mu \circ \phi ^{-1}$ on $X$; that is,
\begin{equation*}
\nu (A):=\mu (\phi ^{-1}(A))
\end{equation*}%
for all Borel sets $A$ contained in $X$. \ Let $\pounds ^{1}(\mu ):=\{f:f$
is Borel function on $Z$ such that $\ \int \left\vert f\right\vert d\mu
<\infty \}$, and let $L^{1}(\mu ):=\{[f]:f\in \pounds ^{1}(\mu )\}$, where $%
[f]:=\{g\in \pounds ^{1}(\mu ):\int \left\vert f-g\right\vert d\mu =0\}$. \
The map%
\begin{equation*}
\psi \rightarrow \int_{Z}(\psi \circ \phi )fd\mu \; \quad (f \in L^1(\mu))
\end{equation*}%
defines a bounded linear functional on$\ L^{\infty }(\nu ).$ \ If attention
is restricted to characteristic functions $\chi _{A}$ in $L^{\infty }(\nu )$%
,
\begin{equation*}
A\rightarrow \int_{Z}(\chi _{A}\circ \phi )fd\mu =\int_{\phi ^{-1}(A)}fd\mu
\end{equation*}%
is a countably additive measure defined on Borel sets in $X$, \ that is
absolutely continuous with respect to$\ \nu $. \ Hence there is a unique
element $E(f)$ in $L^{1}(\nu )$ such that
\begin{equation*}
\int_{Z}(\chi _{A}\circ \phi )fd\mu =\int_{X}\chi _{A}E(f)d\nu
\end{equation*}%
for all Borel subsets $A$ of $X.$ \ By an approximation argument one can
show that
\begin{equation*}
\int_{Z}(\psi \circ \phi )fd\mu =\int_{X}\psi E(f)d\nu
\end{equation*}%
for all $\psi $ in $L^{\infty }(\nu ).$ \ This defines a map
\begin{equation*}
E:\pounds ^{1}(\mu )\rightarrow L^{1}(\nu )
\end{equation*}%
called the \textit{expectation operator}.

\begin{definition}
\label{Def1} A disintegration of the measure $\mu $ with respect to $\phi $
is a function $t\mapsto \Phi _{t}$ from $X$ to $M(Z),$ such that\newline
(i) \ for each $t$ in $X,$ $\Phi _{t}$ is a probability measure;\newline
(ii)\ \ if $f\in \pounds ^{1}(\mu ),$ $E(f)(t)=\int_{Z}fd\Phi _{t}$ a.e. $%
[\nu ].$
\end{definition}

In the sequel we always assume that $X$ and $Y$ are compact metric spaces;
in case reference is made to weighted shifts, we assume that $X,Y\subseteq
\mathbb{R}_{+}$.

\begin{definition}
\label{defmarg}Given a measure $\mu $ on $X\times Y$, the marginal measure $%
\mu ^{X}$ is given by $\mu ^{X}:=\mu \circ \pi _{X}^{-1}$, where $\pi
_{X}:X\times Y\rightarrow X$ is the canonical projection onto $X$. \ Thus, $%
\mu ^{X}(E)=\mu (E\times Y)$, for every $E\subseteq X$. \ Observe that if $%
\mu $ is a probability measure, then so is $\mu ^{X}$.
\end{definition}

Then, we have:

\begin{theorem}
\label{thm:disintegration} \cite{Yoo} Let $\mu $ and $\Phi _{t}$ be as in
Definition \ref{Def1}. \ Define
\begin{equation*}
\mu^{Y}(G):=\mu (X\times G)\text{ }(\text{all }G\subseteq Y)\text{.}
\end{equation*}%
Then for almost every $t$ in $Y$ $($with respect to $\mu _{Y})$ and all
continuous functions $\phi (s,t)$ on the product space $X\times Y,$ the
measures $\mu ^{Y}$ and $\Phi _{t}$ satisfy%
\begin{equation}
\int_{X\times Y}\phi d\mu (s,t)=\int_{Y}(\int_{X}\phi (s,t)d\Phi
_{t}(s))d\mu ^{Y}(t).  \label{5}
\end{equation}
\end{theorem}

\begin{theorem}
\label{DOM2} \cite{CuYo2} Let $\mu $ be the Berger measure of a subnormal $2$%
-variable weighted shift, and for $j\geq 0$ let $\xi _{j}$ be the Berger
measure of the associated $j$-th horizontal $1$-variable weighted shift $%
W_{j}=W_{\alpha ^{(j)}}$. \ Then\ $\xi _{j}=\mu _{j}^{X}$ (cf. Definition %
\ref{defmarg}), where $d\mu _{j}(s,t):=\frac{1}{\gamma _{0j}}t^{j}d\mu (s,t)$%
; more precisely,
\begin{equation}
d\xi _{j}(s)=\{\frac{1}{\gamma _{0j}}\int_{Y}t^{j}\;d\Phi _{s}(t)\}d\mu
^{X}(s),  \label{Marginal measure}
\end{equation}%
where $d\mu (s,t)\equiv d\Phi _{s}(t)d\mu ^{X}(s)$ is the disintegration of $%
\mu $ by vertical slices. \ A similar result holds for the Berger measure $%
\eta _{i}$ of the associated $i$-th vertical $1$-variable weighted shifts $%
W_{\beta ^{(i)}}\;(i\geq 0)$.
\end{theorem}


\bigskip
\noindent {\bf  Acknowledgment}. \ 
Some of the calculations in this paper were made with the software tool {\it Mathematica} \cite{Wol}.






\end{document}